\documentclass[a4paper,10pt]{amsart}
\pdfoutput=1
\usepackage[utf8]{inputenc}
\usepackage{amsmath,amsthm,amssymb,mathrsfs,thmtools}
\usepackage{dsfont}
\newtheoremstyle{mystyle}{}{}{}{}{\bf}{}{\newline}{}
\theoremstyle{mystyle}

\usepackage{enumitem}
\usepackage[toc,page]{appendix}
\usepackage{mathtools,array}

\usepackage{graphicx,lipsum}
\usepackage{fancyhdr,floatpag}
\usepackage{upgreek, color}
\usepackage[colorlinks=false,bookmarks=true]{hyperref}
\usepackage[usenames, dvipsnames]{xcolor}
\usepackage[style=numeric,backend=bibtex,maxnames=5]{biblatex}
\usepackage{blkarray}
\usepackage{multirow}
\usepackage{pdfpages}

\hypersetup{linkbordercolor={0 0.60 0.75}, linkcolor=MidnightBlue}
\hypersetup{citebordercolor={0 0.60 0.75}, citecolor=MidnightBlue}

\renewbibmacro{in:}{}
\numberwithin{equation}{section}
\theoremstyle{plain}

\newtheorem{thm}{Theorem}[section]

\newtheorem{cor}[thm]{Corollary}
\newtheorem{dfn}[thm]{Definition}
\newtheorem{lem}[thm]{Lemma}

\newtheorem{pps}[thm]{Proposition}

\theoremstyle{mystyle}
\newtheorem{rmk}[thm]{Remark}
\theoremstyle{mystyle}
\newtheorem{examplex}[thm]{Example}
\newenvironment{exm}
  {\pushQED{\qed}\renewcommand{\qedsymbol}{$\triangle$}\examplex}
  {\popQED\endexamplex}

\declaretheoremstyle[
  spaceabove=-6pt,
  spacebelow=6pt,
  headfont=\normalfont\bfseries,
  postheadspace=1em,
  qed=\qedsymbol,
  headpunct={}
]{mystyle} 
\declaretheorem[name={Proof},style=mystyle,unnumbered,
]{Proof}

\makeatletter
\renewenvironment{proof}[1][\proofname] {\par\pushQED{\qed}\normalfont\topsep6\p@\@plus6\p@\relax\trivlist\item[\hskip\labelsep\bfseries#1\@addpunct{.}]\ignorespaces}{\popQED\endtrivlist\@endpefalse}
\makeatother

\newcommand{\Z}{\mathds{Z}}
\newcommand{\C}{\mathds{C}}
\newcommand{\Q}{\mathds{Q}}
\newcommand{\F}{\mathds{F}}
\newcommand{\N}{\mathds{N}}

\newcommand{\G}{\mathbf{G}}
\newcommand{\GLam}{\mathbf{G}_{\Lambda}}
\newcommand{\m}{\mathbf{m}}
\newcommand{\p}{\mathfrak{p}}

\newcommand{\ri}{\mathcal{O}}
\newcommand{\lri}{\mathfrak{o}}

\newcommand{\g}{\mathfrak{g}}

\newcommand{\lrip}{\lri/\p^N}
\newcommand{\ric}{\ri_{\p}}
\newcommand{\W}{W^{\lri}}

\newcommand{\z}{\mathfrak{z}}

\newcommand{\pint}{\mathscr{Z}}
\newcommand{\cczeta}{\zeta^\textup{cc}}
\newcommand{\rzeta}{\zeta^\textup{irr}}
\newcommand{\rtzeta}{\zeta^{\widetilde{\textup{irr}}}}
\newcommand{\kzeta}{\zeta^\textup{k}}
\newcommand{\rg}{\widetilde{r}}

\newcommand{\lvlzf}{\mathcal{Z}}
\newcommand{\rlvlzf}{\mathcal{Z}^\textup{irr}}
\newcommand{\clvlzf}{\mathcal{Z}^\textup{cc}}
\newcommand{\astlvlzf}{\lvlzf^{\ast}}
\newcommand{\astzeta}{\zeta^{\ast}}

\newcommand{\rk}{\textup{rk}}

\newcommand{\vry}{\nu(\mathcal{R}(\mathbf{y}))}

\newcommand{\No}{\mathcal{N}_{B,\m}(\lrip)}

\newcommand{\ua}{u_{A_{\Lambda}}}
\newcommand{\ub}{u_{B_{\Lambda}}}

\newcommand{\komplement}{\textup{c}}
\newcommand{\cc}{{c}}

\newcommand{\irrup}{\textup{irr}}
\newcommand{\ccup}{\textup{cc}}

\newcommand{\specp}{\textup{Spec}(\ri)\setminus\{(0)\}}

\newcommand{\Frac}{\textup{Frac}}
\newcommand{\commF}{A_{\mathcal{F}_{n,\delta}}(\underline{X})}
\newcommand{\commG}{A_{\mathcal{G}_{n}}(\underline{X})}
\newcommand{\commH}{A_{\mathcal{H}_{n}}(\underline{X})}
\newcommand{\commGen}{A_{\Lambda}(\underline{X})}
\newcommand{\commA}{A_{\Lambda}}

\newcommand{\commFm}{A_{\mathcal{F}_{n,\delta}}^{(m)}(\underline{X})}
\newcommand{\commGm}{A_{\mathcal{G}_{n}}^{(m)}(\underline{X})}
\newcommand{\commHm}{A_{\mathcal{H}_{n}}^{(m)}(\underline{X})}
\newcommand{\commGenm}{A_{\Lambda}^{(m)}(\underline{X})}
\newcommand{\Fnd}{\mathcal{F}_{n,\delta}}
\newcommand{\Gn}{\mathcal{G}_{n}}
\newcommand{\Hn}{\mathcal{H}_{n}}

\newcommand{\DL}{\textup{D}_{\Lambda}}
\newcommand{\wL}{\omega_{\Lambda}}
\newcommand{\vLm}{\nu^{\Lambda}_m}
\newcommand{\vm}{\nu_{m}}
\newcommand{\CLm}{\mathfrak{C}^{\Lambda}_{m}}
\newcommand{\kLm}{k^{\Lambda}_{m}}
\newcommand{\kFm}{k^{\Fnd}_{m}}
\newcommand{\kGm}{k^{\Gn}_{m}}
\newcommand{\kHm}{k^{\Hn}_{m}}
\newcommand{\diag}{\textup{Diag}}
\newcommand{\col}{\mathcal{C}}
\newcommand{\coli}{\mathcal{C}^{i}}
\newcommand{\colj}{\mathcal{C}^{j}}
\newcommand{\row}{\mathcal{R}}

\makeatletter
\newcommand{\thickhline}{
    \noalign {\ifnum 0=`}\fi \hrule height 1pt
    \futurelet \reserved@a \@xhline
}
\newcolumntype{"}{@{\hskip\tabcolsep\vrule width 1pt\hskip\tabcolsep}}
\makeatother

\newcommand\coolover[2]{\mathrlap{\smash{\overbrace{\phantom{
    \begin{matrix} #2 \end{matrix}}}^{\mbox{$#1$}}}}#2}

\newcommand\coolleftbrace[2]{
#1\left\{\vphantom{\begin{matrix} #2 \end{matrix}}\right.}

\makeatletter
\renewcommand*\env@matrix[1][*\c@MaxMatrixCols c]{
  \hskip -\arraycolsep
  \let\@ifnextchar\new@ifnextchar
  \array{#1}}
\makeatother

\newcommand\Tstrut{\rule{0pt}{2.6ex}}         
\newcommand\Bstrut{\rule[-0.9ex]{0pt}{0pt}}  

\usepackage[labelsep=period]{caption}
\DeclareMathOperator{\Mat}{Mat}

\DeclareMathOperator{\class}{k}

\DeclareMathOperator{\des}{des}
\DeclareMathOperator{\rmaj}{rmaj}
\DeclareMathOperator{\maj}{maj}
\DeclareMathOperator{\invb}{inv}
\DeclareMathOperator{\negb}{neg}
\DeclareMathOperator{\nspb}{nsp}

\makeatletter
\renewcommand{\@secnumfont}{\bfseries}
\makeatother
\usepackage{etoolbox}
\makeatletter
\patchcmd{\section}{\scshape}{\bf}{}{}
\patchcmd{\subsubsection}{\scshape}{\bf}{}{}
\makeatother

\title[paper]{Bivariate representation and conjugacy class zeta functions associated to unipotent group schemes, II: Groups of type $F$, $G$, and $H$ \vspace{-3ex}}
\author{Paula Macedo Lins de Araujo}
\address{Fakult\"at f\"ur Mathematik, Universit\"at Bielefeld, Germany}
\curraddr{Wiskunde, KU Leuven, Campus Kulak Kortrijk, Etienne Sabbelaan 53 bus 7657, 8500 Kortrijk, Belgium}
\email{paula.macedolinsdearaujo@kuleuven.be}
\date{\today}

\begin{document}

\clearpage
\pagenumbering{arabic} 
\begin{abstract}
This is the second of two papers introducing and investigating two bivariate zeta functions associated to unipotent group schemes over rings of integers of number fields.
In the first part, we proved some of their properties such as rationality and functional equations.
Here, we calculate such bivariate zeta functions of three infinite families of nilpotent groups of class $2$ generalising the Heisenberg group of $3\times 3$-unitriangular matrices 
over rings of integers of number fields.
The local factors of these zeta functions are also expressed in terms of sums over finite hyperoctahedral groups, which provides formulae for joint distributions of 
three statistics on such groups. 
\end{abstract}
\maketitle \vspace{-1.0cm}
\thispagestyle{empty}
\section{Introduction and statement of main results}\label{intro}
In the first part~\cite{PL18} of this work, we introduced bivariate zeta functions of groups associated to unipotent group schemes over rings of integers of numbers fields. In this second part, we provide explicit examples of such zeta functions for infinite families of nilpotent groups of class $2$, and use these expressions to provide formulae for joint distributions of three statistics on hyperoctahedral groups. 

We are interested in understanding the following data of a group $G$.
\begin{align*}
  r_n(G)&=|\{\text{isomorphism classes of $n$-dimensional irreducible complex}\\
  &\phantom{r(G)}\text{representations of } G\}|,\\
  c_n(G)&=|\{\text{conjugacy classes of $G$ of cardinality }n\}|.
\end{align*}
If all these numbers are finite---for instance, if $G$ is a finite group---we define the following zeta functions.
\begin{dfn}\label{univa} The \emph{representation} and the \emph{conjugacy class zeta functions} of the group $G$ are, respectively, 
\begin{align*}
  \rzeta_{G}(s)=\sum_{n=1}^{\infty}r_n(G)n^{-s} \vspace{0.7cm} \text{ and }\hspace{0.5cm}
  \cczeta_{G}(s)=\sum_{n=1}^{\infty}\cc_n(G)n^{-s},
\end{align*}
where $s$ is a complex variable. 
\end{dfn}

Let $K$ denote a number field and $\ri$ its ring of integers. 
Let $\G$ be a unipotent group scheme over $\ri$. The group $\G(\ri)$ is a finitely generated, torsion-free nilpotent group ($\mathcal{T}$-group for short)---
see~\cite[Section~2.1.1]{StVo14}---whereas, for a nonzero ideal $I$ of $\ri$, the group $\G(\ri/I)$ is a finite group.
For $\mathcal{T}$-groups, the numbers $r_n(G)$ and $c_n(G)$ are not all finite. 
We thus define bivariate zeta functions which count such data for the principal congruence quotients~$\G(\ri/I)$.

\begin{dfn}\label{biva} The \emph{bivariate representation} and the \emph{bivariate conjugacy class zeta functions} of $\G(\ri)$ are, respectively,
\begin{align*} 
\rlvlzf_{\G(\ri)}(s_1,s_2)&=\sum_{(0) \neq I \unlhd \ri }\rzeta_{\G(\ri/I)}(s_1)|\ri:I|^{-s_2}\text{ and}\\
\clvlzf_{\G(\ri)}(s_1,s_2)&=\sum_{(0) \neq I \unlhd \ri }\cczeta_{\G(\ri/I)}(s_1)|\ri:I|^{-s_2},
\end{align*}
where $s_1$ and $s_2$ are complex variables. 
\end{dfn}
These series converge for $s_1$ and $s_2$ with sufficiently large real parts---cf. \cite[Proposition~2.3]{PL18}---and satisfy the following Euler decompositions:
\begin{equation}\label{Euler}\astlvlzf_{\G(\ri)}(s_1,s_2)=\prod_{\p \in \specp} \astlvlzf_{\G(\ric)}(s_1,s_2),\end{equation}
where $\ast\in\{\irrup,\ccup\}$ and the completion of $\ri$ at the nonzero prime ideal~$\p$ is denoted by~$\ric$. When considering a fixed prime ideal $\p$, we write simply $\ric=\lri$
and $\G_N:=\G(\lrip)$. With this notation, the local factor at $\p$ is given by 
\begin{align}\label{localfactors}\astlvlzf_{\G(\ric)}(s_1,s_2)=\astlvlzf_{\G(\lri)}(s_1,s_2)= \sum_{N=0}^{\infty}\astzeta_{\G_N}(s_1)|\lri:\p|^{-Ns_2}.\end{align}

In this part of the work, we calculate explicitly such local bivariate zeta functions of three infinite families of nilpotent groups of class $2$.
Such groups are constructed from the following class-$2$-nilpotent $\Z$-Lie lattices, that is, free and finitely generated $\ri$-modules together with an anti-symmetric bi-additive 
form $[~,~]$ which satisfies the Jacobi identity.
\begin{dfn}\label{fgh} For $n\in\N$ and $\delta\in\{0,1\}$, consider the nilpotent $\Z$-Lie lattices 
\begin{align*} \Fnd= \langle x_k,y_{ij} \mid ~&  [x_i,x_j]-y_{ij}, 1 \leq k \leq 2n+\delta, 1\leq i <j \leq 2n+\delta\rangle, \\
\Gn= \langle x_k,y_{ij}  \mid ~& [x_i,x_{n+j}]-y_{ij}, 1 \leq k \leq 2n, 1\leq i ,j \leq n \rangle,\\
\Hn= \langle x_k,y_{ij} \mid ~&  [x_i,x_{n+j}]-y_{ij},[x_j,x_{n+i}]-y_{ij},~1 \leq k \leq 2n, 1\leq i\leq j \leq n\rangle.
\end{align*}
\end{dfn}
By convention, relations that do not follow from the given ones are trivial.

Let $\Lambda$ be one of the $\Z$-Lie lattices of Definition~\ref{fgh}.
We consider the unipotent group scheme $\GLam$ associated to $\Lambda$ obtained by the construction of~\cite[Section~2.4]{StVo14}.
Following~\cite{StVo14}, these unipotent group schemes are denoted by $F_{n,\delta}$, $G_n$, and $H_n$, and groups of the form $F_{n,\delta}(\ri)$, $G_n(\ri)$, and $H_n(\ri)$ are called groups of type $F$, $G$, and $H$, respectively.

The unipotent group schemes $F_{n,\delta}$, $G_n$, and $H_n$ provide different generalisations of the Heisenberg group scheme $\mathbf{H}=F_{1,0}=G_1=H_1$, where $\mathbf{H}(\ri)$ 
is the Heisenberg group of upper uni-triangular $3\times 3$-matrices over~$\ri$. 
The interest in such $\Z$-Lie lattices arises from their very construction. Roughly speaking, their defining relations reflect the reduced, irreducible, prehomogeneous vector spaces of, respectively, complex $n\times n$ antisymmetric matrices, complex $n\times n$-matrices and complex $n\times n$ symmetric matrices---here, 
the relative invariants are given by $\textup{Pf}$, $\det$ and $\det$, respectively, where $\textup{Pf}(X)$ denotes the Pfaffian of an antisymmetric matrix $X$. 
We refer the reader to~\cite[Section~6]{StVo14} for details.

For the rest of this paper, $\Lambda$ is one of the $\Z$-Lie lattices of Definition~\ref{fgh} and $\G=\GLam$ denotes the unipotent group schemes associated to $\Lambda$.
\subsection{Bivariate conjugacy class and class number zeta functions}
Our first result concerns bivariate conjugacy class zeta functions, which leads to similar results for (univariate) class number zeta functions.
\begin{thm}\label{cclvl} Let $n \in \N$, and $\delta \in \{0,1\}$. Then, for each nonzero prime ideal~$\p$ of~$\ri$ with $q=|\ri:\p|$, 
\[\clvlzf_{F_{n,\delta}(\lri)}(s_1,s_2)= \frac{1-q^{\binom{2n+\delta-1}{2}-(2n+\delta-1)s_1-s_2}}{(1-q^{\binom{2n+\delta}{2}-s_2})
(1-q^{\binom{2n+\delta}{2}+1-(2n+\delta-1)s_1-s_2})}.\]
Write $q^{-s_1}=T_1$ and $q^{-s_2}=T_2$. For $n\geq 2$, 
\begin{align*}&\clvlzf_{G_n(\lri)}(s_1,s_2)=\\
&\frac{(1-q^{2\binom{n}{2}}T_{1}^{n}T_2)(1-q^{2\binom{n}{2}+1}T_{1}^{2n-1}T_2)+q^{n^2}T_{1}^{n}T_2(1-q^{-n})(1-q^{-(n-1)}T_{1}^{n-1})}
{(1-q^{n^2}T_2)(1-q^{n^2}T_{1}^{n}T_2)(1-q^{n^2+1}T_{1}^{2n-1}T_2)},\\
&\clvlzf_{H_n(\lri)}(s_1,s_2)=\\
&\frac{(1-q^{\binom{n}{2}}T_{1}^{n}T_2)(1-q^{\binom{n}{2}+2}T_{1}^{2n-1}T_2)+q^{\binom{n+1}{ 2}}T_{1}^{n}T_2(1-q^{-n+1})(1-q^{-(n-1)}T_{1}^{n-1})}
{(1-q^{\binom{n+1}{2}}T_2)(1-q^{\binom{n+1}{2}+1}T_{1}^{n}T_2)(1-q^{\binom{n+1}{2}+1}T_{1}^{2n-1}T_2)}.\\
\end{align*}
\end{thm}
The proof of Theorem~\ref{cclvl} is given in Section~\ref{pcclvl}.

Let $\class(G)$ denote the \emph{class number} of the finite group $G$, that is, the number of conjugacy classes or, equivalently, the number of irreducible complex 
characters of~$G$. In particular, wee see that $\rzeta_G(0)=\cczeta_G(0)=\class(G)$.
We recall from~\cite[Section~1.2]{PL18} that, for a unipotent group scheme $\G$, one can obtain the (univariate) class number zeta function $\kzeta_{\G(\ri)}(s)$ via the following 
specialisation:
\begin{equation}\label{specialisation}\rlvlzf_{\G(\ri)}(0,s)=\clvlzf_{\G(\ri)}(0,s)=\sum_{(0) \neq I \unlhd \ri}\class(\G(\ri/I))|\ri:I|^{-s}=:\kzeta_{\G(\ri)}(s),\end{equation}
where $s$ is a complex variable. We remark that the term `conjugacy class zeta function' is sometimes used for what we call `class number zeta function'; see for instance~\cite{BDOP,ro17,duSau05}.
Specialisation~\eqref{specialisation} applied to Theorem~\ref{cclvl} gives the following.

\begin{cor}\label{kzf} For all $n \in \N$ and $\delta \in \{0,1\}$,
\begin{equation}
\label{kzetaF}\kzeta_{F_{n,\delta}(\ri)}(s)=\frac{\zeta_K(s-\binom{2n+\delta}{2}-1)\zeta_{K}(s-\binom{2n+\delta}{2})}{\zeta_K(s-\binom{2n+\delta-1}{2})},
\end{equation}
where $\zeta_K(s)$ is the Dedekind zeta function of the number field $K=\Frac(\ri)$.
Furthermore, for $n\geq 2$, the class number zeta functions of $G_n(\ri)$ and $H_n(\ri)$ are
  \begin{align*}
    \kzeta_{G_n(\ri)}(s)&=\prod_{\p \in \specp}\hspace{-0.2cm}
    \frac{(1-q_{\p}^{2\binom{n}{2}-s})(1-q_{\p}^{2\binom{n}{2}+1-s})+q_{\p}^{n^2-s}(1-q_{\p}^{-n})(1-q_{\p}^{-n+1})}{(1-q_{\p}^{n^2-s})^2(1-q_{\p}^{n^2+1-s})},&\\
    \kzeta_{H_n(\ri)}(s)&=\prod_{\p \in \specp}\hspace{-0.2cm}
    \frac{(1-q_{\p}^{\binom{n}{2}-s})(1-q_{\p}^{\binom{n}{2}+2-s})+q_{\p}^{\binom{n+1}{2}-s}(1-q_{\p}^{-n+1})^2}{(1-q_{\p}^{\binom{n+1}{2}-s})(1-q_{\p}^{\binom{n+1}{2}+1-s})^2},&
  \end{align*}
where $q_{\p}=|\ri:\p|$, for each $\p \in \specp$.
\end{cor}
In particular, all local factors of the bivariate conjugacy class zeta functions of groups of type $F$, $G$, and $H$ are rational in $q_{\p}$, $q_{\p}^{-s_1}$, and $q_{\p}^{-s_2}$, 
whilst all local factors of their class number zeta functions are rational in $q_{\p}$ and $q_{\p}^{-s}$. Moreover, the local factors of both zeta functions satisfy 
functional equations. This generalises~\cite[Theorem~1.4]{PL18} for these groups. 

Formula~\eqref{kzetaF} is also shown in~\cite{ro17}; it is a consequence of both~\cite[Proposition~5.11 and Proposition~6.4]{ro17}; 
see Remarks~\ref{sod} and~\ref{minimality}. 
In~\cite[Section~4.2]{PL18}, we write the bivariate zeta functions of Definition~\ref{biva} in terms of $\p$-adic integrals. These integrals under specialisation~\eqref{specialisation} 
coincide with the integrals~\cite[(4.3)]{ro17} under the specialisation of the \emph{ask zeta function} to the class number zeta function given in~\cite[Theorem~1.7]{ro17}; 
cf.~\cite[Remark~4.10]{PL18}.

\subsection{Bivariate representation and twist representation zeta functions}
To state our next result, we introduce some notation.

Let $X,Y$ denote indeterminates in the field $\Q(X,Y)$. Given $n \in \N$, set \[(\underline{n})_{X}= 1-X^{n} \hspace{0.2cm}\text{ and }\hspace{0.2cm}
(\underline{n})_{X}!= (\underline{n})_{X}(\underline{n-1})_{X} \dots (\underline{1})_{X}.\] 
For $a,b \in \N_0$ such that $a\geq b$, the $X$-\emph{binomial coefficient} of $a$ over $b$ is
\[\binom{a}{b}_{X}= \frac{(\underline{a})_X}{(\underline{b})_{X}(\underline{a-b})_{X}}.\]

Given $n \in \N$, set $[n]=\{1,\dots, n\}$ and $[n]_0=[n]\cup\{0\}$. 
Given a subset $\{i_1, \dots , i_l \}\subset \N$, we write $\{i_1, \dots , i_l \}_{<}$ meaning that $i_1< i_2 < \dots < i_l$.
For $I=\{i_1, \dots , i_l \}_{<} \subseteq [n-1]_0$, let $\mu_j:=i_{j+1}-i_j$ for all $j \in [l]_0$, where $i_0=0$, $i_{l+1}=n$, and define 
\[\binom{n}{I}_X=\binom{n}{i_l}_X\binom{i_l}{i_{l-1}}_X \dots \binom{i_2}{i_1}_X.\]
The $Y$-\emph{Pochhammer symbol} is defined as
\[(X;Y)_{n}=\prod_{i=0}^{n-1}(1-XY^i).\]

\begin{thm}\label{rlvl}
 Let $n \in \N$ and $\delta \in \{0,1\}$. Then, for each nonzero prime ideal~$\p$ of~$\ri$ with $|\ri:\p|=q$, 
 \begin{align*}
  \rlvlzf_{\G(\lri)}(s_1,s_2)&=\frac{1}{1-q^{\bar{a}(\G,n)-s_2}} \sum_{I \subseteq [n-1]_0}f_{\G,I}(q^{-1})\prod_{i \in I}
  \frac{q^{\bar{a}(\G,i)-(n-i)s_1-s_{2}}}{1-q^{\bar{a}(\G,i)-(n-i)s_1-s_{2}}},
 \end{align*}
where  $f_{\G,I}(X)$ and $\bar{a}(\G,i)$, for all $I=\{i_1, \dots, i_l\}_{<} \subseteq [n-1]_0$ and for all $i\in [n]_0$, are defined in Table~\ref{Tablerlvl}.
\begin{table*}[h]\caption {Terms appearing in Theorem~\ref{rlvl}}\label{Tablerlvl}
	\centering
		\begin{tabular}{ c | c | c }
		  $\G$ & $f_{\G,I}(X)$ & $\bar{a}(\G,i)$ \\[0.5ex]
		  \hline
		  $F_{n,\delta}$ & $\binom{n}{I}_{X^2}(X^{2(i_1+\delta)+1};X^2)_{n-i_1}$ &  $\binom{2n+\delta}{2}-\binom{2i+\delta}{2}+2i+\delta$ \Tstrut\Bstrut\\
		  
		  $G_n$ & $\binom{n}{I}_{X}(X^{i_1+1};X)_{n-i_1}$ & $n^2-i^{2}+2i$ \Tstrut\Bstrut\\[1ex]
		  
		  $H_n$ & $\left( \prod_{j=1}^{l}(X^2;X^2)_{\lfloor \mu_j/2 \rfloor}^{-1} \right)(X^{i_1+1};X)_{n-i_1}$ & $\binom{n+1}{ 2}-\binom{i+1}{2}+2i$\Tstrut\Bstrut\\
		\end{tabular}
\end{table*}
\end{thm}
The proof of Theorem~\ref{rlvl} may be found in Section~\ref{prlvl}. 

The numbers $\bar{a}(\G,i)$ are slight modifications of the numbers $a(\G,i)$ given in \cite[Theorem~C]{StVo14},
namely $\bar{a}(F_{n,\delta},i)=a(F_{n,\delta},i)+2i+\delta$ and $\bar{a}(\G,i)=a(\G,i)+2i$, for $\G(\ri)$ of type $G$ and $H$.

Let $\G$ be a unipotent group scheme over~$\ri$ such that $\G(\ri)$ has nilpotency class~$2$. The local factors of the bivariate representation zeta functions of $\G(\ri)$ 
specialise to the local factors of its twist representation zeta function~$\rtzeta_{\G(\ri)}(s)$. 
This is a representation zeta function encoding the numbers $\rg_n(G)$ of $n$-dimensional twist-isoclasses of irreducible complex representations of $G$. Twist isoclasses are the equivalence classes of the relation on the set of irreducible complex representations of~$G$ given by $\rho \thicksim \sigma$ if and only if there exists a $1$-dimensional representation~$\chi$ of~$G$ such that $\rho \cong \chi \otimes \sigma$. 
In the context of topological groups, only continuous representations are considered. 

Twist representation zeta functions of $\mathcal{T}$-groups are investigated, for instance, in~\cite{DuVo14, StVo14, StVo17}. 
In~\cite{Vo10}, Voll introduced a method to compute explicitly various zeta functions associated to groups and rings and deduced local functional equations from the obtained formulae. One of the applications presented in this paper is that almost all local factors of twist representation zeta functions of $\mathcal{T}$-groups have functional equations. In~\cite[Theorem~1.5]{HrMaRi}, the authors prove that \emph{all} local factors of these zeta functions are rational functions. This is a Corollary of~\cite[Theorem~1.3]{HrMaRi}, which asserts that multivariate zeta functions counting the equivalence classes in some uniformly definable family of equivalence relations are rational functions. 
In \cite[Table~1]{Ro18comp}, by implementing his methods in Zeta~\cite{RoZeta}, Rossmann provides fomulae for almost all local factors of twist representation zeta functions associated with all unipotent algebraic groups of dimension at most~$6$ over an arbitrary number field. 

In~\cite{PL18}, we show how to specialise bivariate representation zeta functions to twist representation zeta functions. 
 More precisely, we show that for a group~$\G(\ri)$ of nilpotency class~2, there is a constant $r=r(\G)$ such that 
\begin{equation}\label{specialisationirr}\prod_{\p \in \specp}\left((1-q^{r-s_2})\rlvlzf_{\G(\ric)}(s_1,s_2)\mid_{\substack{s_1\to s-2\\ s_2\to r\phantom{-2}}}\right)=
\rtzeta_{\G(\ri)}(s).
\end{equation}
Since groups of type $F$, $G$, and $H$ are $\mathcal{T}$-groups of class $2$, we obtain the formulae of~\cite[Theorem~C]{StVo14} by applying specialisation~\eqref{specialisationirr} to Theorem~\ref{rlvl}. 
\subsection{Joint distributions on hyperoctahedral groups}
The polynomials $f_{\G,I}(X)$ appearing in Theorem~\ref{rlvl} can be expressed in terms of distributions of statistics on Weyl groups of type $B$, also called hyperoctahedral groups~$B_n$; see Sections~\ref{Weyl} and~\ref{stat}. These are the groups~$B_n$ of permutations~$w$ of the set $[\pm n]_0=\{-n, \dots, -1\} \cup [n]_0$ such that $w(-i)=-w(i)$ for all $i \in [\pm n]_0$.

In Lemma~\ref{weylfgh}, we describe local bivariate representation zeta functions of groups $\G(\ri)$ as sums over $B_n$ in terms of statistics on such groups.
As the local factors of the bivariate representation and the bivariate conjugacy class zeta functions of $\G(\ri)$ specialise to its class number zeta function,
the formulae in terms of statistics on hyperoctahedral groups $B_n$ can be compared with the formulae of Corollary~\ref{kzf}, leading to formulae for the joint distribution of 
three functions on Weyl groups of type $B$; see Propositions~\ref{statF} and~\ref{statGH}.

More precisely, the formulae of Lemma~\ref{weylfgh} under specialisation~\eqref{specialisation} provide a formula of the following form for the class number zeta function of 
$\G(\lri)$:
\[\kzeta_{\G(\lri)}(s)=\frac{\sum_{w \in B_n}\chi_{\G}(w)q^{-\overline{h_\G}(w)-\des(w)s}}{\prod_{i=0}^{n}(1-q^{\bar{a}(\G,i)-s})},\]
where $\chi_\G$ is one of the linear characters $(-1)^{\negb}$ or $(-1)^{\ell}$ of $B_n$, where $\negb(w)$ denotes the number of negative entries of $w$, and $\ell$ is the 
standard Coxeter length function of~$B_n$. 
Moreover, the functions $\overline{h_{\G}}$ are sums of statistics on $B_n$ for each $\G$ and $\des(w)$ is the cardinality of the descent set of $w \in B_n$; 
see Section~\ref{Weyl} for definitions.
\subsection{Local functional equations}
The formulae for the bivariate zeta functions given in Theorems~\ref{rlvl} and~\ref{cclvl} allow us to strengthen~\cite[Theorem~1.4]{PL18} for groups of type $F$, $G$, and $H$ by 
showing that its conclusion holds for \emph{all} local factors:
\begin{thm}\label{functeq} For $\ast\in\{\irrup,\ccup\}$ and \emph{all} nonzero prime ideal $\p$ of $\ri$ with $|\ri:\p|=q$, the local bivariate zeta function $\mathcal{Z}_{\G(\lri)}^{\ast}(s_1, s_2)$ satisfies the 
functional equation
\begin{equation*}
 \mathcal{Z}_{\G(\lri)}^{\ast}(s_1, s_2)\mid_{q \to q^{-1} }= -q^{h-s_2}\mathcal{Z}_{\G(\lri)}^{\ast}(s_1,s_2),
\end{equation*}
where $h$ is the torsion rank of $\Lambda(\lri)=\Lambda\otimes_{\lri}\lri$; see the exact value of $h$ in Table~\ref{table}.
 \end{thm}
In fact,~\cite[Theorem~1.4]{PL18} states that almost all local factors satisfy functional equations 
of such form, whilst Theorems~\ref{cclvl} and~\ref{rlvl} state that all local factors are given by the same rational functions.
We give an alternative proof for the representation case in Section~\ref{pfuncteq}.
\subsection{Notation}\label{notation}
The following list collects frequently used notation.
\begin{table*}[h]
	\centering
		\begin{tabular}{ c | l }
				$\N$ & $\{1,2,\dots\}$\\
				$\N_0$ & $\{0,1,2\dots\}$\\
				$[n]$ & $\{1,\dots,n\}$ \\
				$[n]_0$ & $\{0,1,\dots,n\}$\\
				$[\pm n]$ & $\{-n,\dots,-1\}\cup [n]_0$\\
				\hline
				$X,Y$ & indeterminates in the field $\Q(X,Y)$\\
				$(\underline{n})_X$ & $1-X^n$, for $n \in \N$\\
				$(\underline{n})_{X}!$& $(\underline{n})_{X}(\underline{n-1})_{X} \dots (\underline{1})_{X}$, for $n \in \N$\\[0.5ex]
				$\binom{a}{b}_{X}$ & $\frac{(\underline{a})_X}{(\underline{b})_{X}(\underline{a-b})_{X}}$, for $a\geq b$ in~$\N$\\[1.5ex]
				$(X;Y)_n$&$\prod_{i=0}^{n-1}(1-XY^i)$\\[0.5ex]
				\hline
				$I=\{i_1, \dots, i_{l}\}_{<}$ & set $I$ of nonnegative integers $i_1<i_2< \dots<i_l$ \\
				$\binom{n}{I}$& $\binom{n}{i_l}\binom{i_l}{i_{l-1}} \dots \binom{i_2}{i_1}$ \\
				$i_0$& $0$ \\
				$i_{l+1}$ & $n$\\
				$\mu_j$ & $i_{j+1}-i_j$, $j \in [l]_0$\\
				\hline
				$K$ & number field\\
				$\ri$& ring of integers of $K$\\
				$\p$& nonzero prime ideal of $\ri$\\
				$\lri=\ric$& completion of $\ri$ at $\p$\\
				$\lri^n$& $n$-fold Cartesian power $\lri\times\dots\times\lri$ \\
				$\p^m$ & $m$th ideal power $\p\cdots\p$ \\
				$\p^{(m)}$& $m$-fold Cartesian power $\p\times\dots\times\p$ \\
				$\textup{Spec}(\ri)$ & set of all prime ideals of~$\ri$\\
				\hline
				$\Lambda$ & one of the $\Z$-Lie lattices $\Fnd$, $\Gn$, or $\Hn$;  see Definition~\ref{fgh}.\\
				$\G$ & one of the unipotent group schemes $F_{n,\delta}$, $G_n$, or $H_n$.\\
				$\commGen$ & $A$-commutator matrix of $\Lambda$; see Definition~\ref{commutator}\\
				$B_{\Lambda}(\underline{Y})$ & $B$-commutator matrix of $\Lambda$; see Definition~\ref{commutator}\\
				\hline
				$v_{\p}$&$\p$-adic valuation\\
				$|.|_{\p}$&$\p$-adic norm\\
				$\|(z_j)_{j\in J}\|_\p$&$\max(|z_j|_\p)_{j\in J}=q^{-v_{\p}((z_i)_{i\in I})}$, $J$ finite index set\\[1ex]
				$W_{k}(\lrip)$&$\{\mathbf{x}\in(\lrip)^k\mid v_{\p}(\mathbf{x})=0\}, ~k\in\N$, $N\in \N_0$\\
				$\W_{k}$&$\{\mathbf{x}\in\lri^k\mid v_{\p}(\mathbf{x})=0\}, ~k \in \N$\\
				\hline
				$\diag(a_{1}, \dots, a_{n})$ & diagonal $n \times n$-matrix $\left(\begin{smallmatrix} a_{1}&&\\&\ddots&\\&&a_{n} \end{smallmatrix}\right)$.
\end{tabular}
\end{table*}
\section{Bivariate zeta functions and \texorpdfstring{$\p$}{p}-adic integrals}\label{padicintegral}
In this section we calculate generic $\p$-adic integrals which are needed in the present paper, and recall 
from~\cite[Proposition~4.8]{PL18} how to write the local bivariate zeta functions of groups of type $F$, $G$, and $H$ in terms of $\p$-adic integrals.

For the rest of Section~\ref{padicintegral}, let $\p$ be a fixed nonzero prime ideal of~$\ri$, and $\lri$ denotes the completion of~$\ri$ at~$\p$. Denote by~$q$ the cardinality of~$\ri/\p$ and by~$p$ its characteristic.

\subsection{Some \texorpdfstring{$\p$}{p}-adic integrals}\label{appadic}
Before we state the first result of this section, we recall some notation. 

Let $z \in \lri$ and let $(z)=\p^{e}\p_{1}^{e_1}\cdots\p_{r}^{e_r}$ be the prime factorisation of the ideal $(z) \lhd \ri$ with $\p_i\neq \p$, for all $i\in[r]$. The $\p$-\emph{adic valuation} of $z$ is given by $v_{\p}(z)=e$, and its $\p$-\emph{adic norm} is $|z|_{\p}=q^{-v_{\p}(z)}$. For a finite index set $J$ and $(z_j)_{j\in J}\in \lri^{J}$, 
define $\|(z_j)_{j\in J}\|_\p=\max(|z_j|_\p)_{j\in J}.$
We denote by $\p^m$ the $m$th ideal power $\p\cdots\p$ and by $\p^{(m)}$ the $m$-fold Cartesian power $\p\times \dots \times \p$.

For each $k \in \N$, consider the set 
\[\W_k:= \lri^k \setminus \p^k= \{ \underline{x} \in \lri^k \mid v_{\p}(\underline{x})=0\}.\]

From now on, let $\mu$ denote the additive Haar measure on $\lri$, normalised so that $\mu(\lri)=1$. We also denote by $\mu$ the product measure on $\lri^n$, for $n \in \N$.

The following is well known; see for instance~\cite[Lemma~2.2.1]{Linsphd}. 
\begin{pps}\label{integral1} Let $r$ be a complex variable. Then, for each $k \in \N$,
 \[\int_{w\in\p^k}|w|_{\p}^{r}d\mu=\frac{q^{-k(r+1)}(1-q^{-1})}{1-q^{-r-1}},\] if the integral on the left-hand side converges absolutely.
\end{pps}

The following lemma is a direct consequence of~\cite[Lemma~5.8]{ro17}, which assures in particular that, for complex variables $r$ and $s$, one has
\begin{equation}\label{Roeq}
 \int_{(y,\underline{x})\in\lri\times\lri^n}|y|_{\p}^{r}\|x_1,\dots, x_n,y\|_{\p}^{s}d\mu=\frac{(1-q^{-1})(1-q^{-r-n-1})}{(1-q^{-r-s-n-1})(1-q^{-r-1})},
\end{equation}
if the integral on the left-hand side converges absolutely.
\begin{pps}\label{integral2} Let $r$ and $s$ be complex variables. 
Then, for each $n \in \N_0$,
  \begin{align*} \int_{(y,\underline{x})\in\p\times\lri^n}\hspace{-0.5cm}|y|_{\p}^{r}\|x_1, \dots, x_n, y\|_{\p}^{s}d\mu&=\frac{(1-q^{-1})(1-q^{-n}+q^{-s-n}-q^{-r-s-n-1})q^{-r-1}}{(1-q^{-r-s-n-1})(1-q^{-r-1})},\\
    \int_{(y,\underline{x})\in\p\times\p^{(n)}}\hspace{-0.5cm}|y|_{\p}^{r}\|x_1, \dots, x_n, y\|_{\p}^{s}d\mu&=\frac{(1-q^{-1})(1-q^{-r-n-1})q^{-r-s-n-1}}{(1-q^{-r-s-n-1})(1-q^{-r-1})},
  \end{align*}
if the integrals on the left-hand side of each equality converge absolutely.
\end{pps}
\begin{proof} 
Since $\p\times\lri^n=\lri\times\lri^n\setminus \W_1\times\lri^n$ and $y \in \W_1$ implies both $|y|_{\p}=1$ and $\|x_1, \dots, x_n,y\|_{\p}=1$, it follows that
\begin{align*} &\int_{(y,\underline{x})\in \p\times\lri^n}\hspace{-0.3cm}|y|_{\p}^{r}\|x_1, \dots, x_n, y\|_{\p}^{s}d\mu
 &=\int_{(y,\underline{x})\in\lri\times\lri^{n}}\hspace{-0.3cm}|y|_{\p}^{r}\|x_1, \dots, x_n, y\|_{\p}^{s}d\mu-\mu(\W_1\times\lri^{n}).
 \end{align*}
The first claim then follows from~\eqref{Roeq} and the fact that $\mu(\W_1\times\lri^{n})=1-q^{-1}$.
Analogously, since $\p\times\p^{(n)}=\p\times\lri^n\setminus\p\times\W_n$, 
 \begin{align*} &\int_{(y,\underline{x})\in\p\times\p^n}|y|_{\p}^{r}\|x_1, \dots, x_n, y\|_{\p}^{s}d\mu\\
 &=\int_{(y,\underline{x})\in \p\times\lri^{n}}|y|_{\p}^{r}\|x_1, \dots, x_n, y\|_{\p}^{s}d\mu-(1-q^{-n})\int_{y\in\p}|y|_{\p}^{r}d\mu.
 \end{align*}
The second claim follows from the first part and Proposition~\ref{integral1}.
\end{proof}
Let $\underline{X}=(X_{11}, \dots, X_{2n})$ be a vector of variables. In the following, consider the matrix 
\begin{equation}\label{matrix}M(\underline{X})=\left[ \begin {array}{cccc} X_{11}&X_{12}&\dots&X_{1n}\\ \noalign{\medskip}X_{21}&X_{22}&\dots&X_{2n}\end {array} \right]\in \Mat_{2\times n}(\lri[\underline{X}]),\end{equation} 
and, for $1\leq i < j \leq n$, write $M_{ij}(\underline{X}):=X_{1i}X_{2j}-X_{1j}X_{2i}$. 

\begin{lem}\label{Lemma:Minors} For $\textbf{\textup{x}}=(1+x_{11}, x_{12}, \dots, x_{2n}) \in \lri^{2n}$, with $x_{ij} \in \p$ for all $i \in [2]$ and $j \in [n]$, we have 
\begin{equation}\label{LemmaMinors}\|\{M_{ij}(\textbf{\textup{x}})\mid~1\leq i<j\leq n\}\|_{\p}=\|\{M_{1i}\mid~i \in \{2, \dots,n\}\}\|_{\p}.\end{equation}
In other words, there is $j \in \{2, \dots, n\}$ such that the minor $M_{1j}(\textbf{\textup{x}})$ has minimal valuation among all $M_{ij}(\textbf{\textup{x}})$. 
\end{lem}
\begin{proof} We only have to show this result for $n=3$ and the general case will follow, as we now explain. Assume that, for all $\textbf{a}=(1+a_{11}, a_{12}, a_{13}, a_{21}, a_{22}, a_{23}) \in \lri^{6}$ with~$a_{ij} \in \p$, the matrix
\begin{equation}\label{Ma}M(\textbf{a})= \left[\begin {array}{cccc} 1+a_{11}&a_{12}&a_{13}\\ \noalign{\medskip}a_{21}&a_{22}&a_{23}\end {array} \right]\end{equation}
has minors satisfying	 $\|M_{12}(\textbf{a}),M_{13}(\textbf{a}),M_{23}(\textbf{a})\|_{\p}=\|M_{12}(\textbf{a}),M_{13}(\textbf{a})\|_{\p}$. Given a matrix of the form~\eqref{matrix} of size $2 \times n$, for arbitrary $n \in \N$, we obtain that for all $1 \leq i < j \leq n$, 
\[\|M_{1i}(\textbf{x}), M_{1j}(\textbf{x}), M_{ij}(\textbf{x})\|_{\p}= \|M_{1i}(\textbf{x}), M_{1j}(\textbf{x})\|_{\p},\]
and therefore~\eqref{LemmaMinors} holds. 

Let us then show that the matrix~\eqref{Ma} satisfies $\|M_{12}(\textbf{a}),M_{13}(\textbf{a}),M_{23}(\textbf{a})\|_{\p}=\|M_{12}(\textbf{a}),M_{13}(\textbf{a})\|_{\p}$.
By definition,
\begin{align*} 	M_{1i}(\textbf{a}) &= (1+a_{11})a_{2i}-a_{1i}a_{21}, \text{ for }i=2,3, \text{ and}\\
								M_{23}(\textbf{a}) &= a_{12}a_{23}-a_{13}a_{22}.
\end{align*}
Denote by $b_{ij}$ the $\p$-adic valuation of $a_{ij}$ for each $i \in [2]$ and $j \in [3]$. Notice that all $b_{ij}$ are greater than zero. Since $v_{\p}(1+a_{11})=0$,  
\[v_{\p}(M_{1i}(\textbf{a}))\geq \min\{b_{2i}, b_{1i}+b_{21}\}, ~i=2,3.\]
Moreover,
\[v_{\p}(M_{23}(\textbf{a})) \geq \min\{b_{12}+b_{23}, b_{13}+b_{22}\}.\]

Choose $\varepsilon_1, \varepsilon_2, \varepsilon_3 \in \N_0$ such that 
\begin{align*} 	v_{\p}(M_{1(i+1)}(\textbf{a}))&= \min\{b_{2(i+1)}, b_{1(i+1)}+b_{21}\}+\varepsilon_i, \text{ for }i \in [2], \text{ and }\\
								v_{\p}(M_{23}(\textbf{a}))&= \min\{b_{12}+b_{23}, b_{13}+b_{22}\}+\varepsilon_3.
\end{align*}
Notice that $\varepsilon_i$ can only be nonzero if $b_{2(i+1)}=b_{1(i+1)}+b_{21}$, for $i=1,2$, and if $b_{12}+b_{23}= b_{13}+b_{22}$, for $i=3$. 

In the following, we divide the proof in five cases, considering values of~$\varepsilon_1$, $\varepsilon_2$ and $v_{\p}(M_{23})$.
In all cases, we show that $v_{p}(M_{23})\geq v_{\p}(M_{1i})$ for some $i \in \{2,3\}$.

\noindent
\textbf{Case~1:} Assume $\varepsilon_1=0$ and $v_{\p}(M_{23}(\textbf{a}))=b_{13}+b_{22}+\varepsilon_3$. Let us show that in this case $v_{\p}(M_{12}(\textbf{a})) \leq v_{\p}(M_{23}(\textbf{a}))$.
 
Suppose that $v_{\p}(M_{12}(\textbf{a}))>v_{\p}(M_{23}(\textbf{a}))$. Then
\[b_{22} \geq v_{\p}(M_{12}(\textbf{a})) > v_{\p}(M_{23}(\textbf{a}))= b_{13}+b_{22}+\varepsilon_3.\]
Consequently,
\[0 > b_{13}+ \varepsilon_3, \]
a contradiction. 

\noindent
\textbf{Case~2:} Assume $\varepsilon_2=0$ and $v_{\p}(M_{23}(\textbf{a}))=b_{12}+b_{23}+\varepsilon_3$. Similar arguments as the ones for the former case show that, in the current case, we have \[v_{\p}(M_{13}(\textbf{a})) \leq v_{\p}(M_{23}(\textbf{a})).\]

\noindent
\textbf{Case~3:} Assume $\varepsilon_1=0$ and $v_{\p}(M_{23}(\textbf{a}))=b_{12}+b_{23}+\varepsilon_3$. 
If $\varepsilon_2=0$, we are back to Case~2, thus we may assume that $\varepsilon_2\neq 0$, in which case $b_{23}=b_{13}+b_{21}$. 
As a consequence, we must have $b_{23}>b_{21}$ and hence
\[v_{\p}(M_{12}(\textbf{a})) \leq b_{12}+b_{21} < b_{12}+b_{23}+\varepsilon_3=v_{\p}(M_{23}(\textbf{a})).\]

\noindent
\textbf{Case~4:} Assume $\varepsilon_2=0$ and $v_{\p}(M_{23}(\textbf{a}))=b_{13}+b_{22}+\varepsilon_3$. Again, we may assume that $\varepsilon_1 \neq 0$. Then, similar arguments as the ones for Case~3 show that 
\[v_{\p}(M_{13}(\textbf{a})) \leq v_{\p}(M_{23}(\textbf{a})).\]

\noindent
\textbf{Case~5:} Assume $\varepsilon_1\neq 0$ and $\varepsilon_2 \neq 0$. 

In this case, we must have $b_{22}=b_{12}+b_{21}$ and $b_{23}=b_{13}+b_{21}$. Hence, $b_{22}-b_{12}=b_{21}=b_{23}-b_{13}$. 
As a consequence, 
\[v_{\p}(M_{23})=b_{12}+b_{23}+\varepsilon_3=b_{13}+b_{22}+\varepsilon_3.\]
If $\varepsilon_3 \geq \varepsilon_1$, then
\[v_{\p}(M_{12}(\textbf{a}))= b_{22}+\varepsilon_1 \leq b_{22}+\varepsilon_3 < b_{13}+b_{22}+\varepsilon_3=v_{\p}(M_{23}(\textbf{a})).\]
If $\varepsilon_3 \geq \varepsilon_2$, then 
\[v_{\p}(M_{13}(\textbf{a})) = b_{23}+\varepsilon_2 \leq b_{23}+\varepsilon_3 < b_{12}+b_{23}+\varepsilon_3=v_{\p}(M_{23}(\textbf{a})).\]
Thus, it suffices to show that $\varepsilon_3 \geq \min\{\varepsilon_1, \varepsilon_2\}$.

Let $u_{ij}$ denote the reduction of $a_{ij}$ modulo $\p^{b_{ij}}$. In particular, $v_{\p}(u_{ij})=0$. By definition
\begin{align*}\varepsilon_{i}&=v_{\p}((1+a_{11})u_{2(i+1)}-u_{1(i+1)}u_{21}), \text{ for }i=1,2, \text{ and }\\
\varepsilon_3&=v_{\p}(u_{12}u_{23}-u_{13}u_{22}).
\end{align*}
Since 
\begin{align*} u_{12}u_{23}-u_{13}u_{22} 	
																					& = u_{12}(u_{23}-u_{13}u_{21})-u_{13}(u_{22}-u_{12}u_{21})\\
																					& = u_{12}((1+a_{11})u_{23}-u_{13}u_{21})-u_{13}((1+a_{11})u_{22}-u_{12}u_{21})\\
																					&-a_{11}(u_{12}u_{23}-u_{13}u_{22}),\\
\end{align*}
it follows that \[\varepsilon_3\geq \min\{\varepsilon_1,\varepsilon_2, v_{\p}(a_{11})+\varepsilon_3\}.\] 
If $\min\{\varepsilon_1,\varepsilon_2, v_{\p}(a_{11})+\varepsilon_3\}=v_{\p}(a_{11})+\varepsilon_3$, we would obtain $v_{\p}(a_{11}) \leq 0$. Therefore, $\varepsilon_3\geq \min\{\varepsilon_1,\varepsilon_2\}$. 
\end{proof}

\begin{pps}\label{integral3} For complex variables $r$ and $s$, the following holds, provided the integral on the left-hand side converges absolutely.
 \begin{align*}&\int_{(y,\underline{x})\in\p\times \W_{2n}}|y|_{\p}^{r}\|\{M_{ij}(\underline{x})\mid~1\leq i<j\leq n\}\cup\{ y\}\|_{\p}^{s}d\mu\\
 &=\frac{(q^n-1)(1-q^{-1})q^{-r-2n-1}}{(1-q^{-1-r})(1-q^{-r-s-n})}\left((q+1)(1-q^{-r-n})q^{-s}+(q^{n}-q)(1-q^{-r-s-n})\right).
  \end{align*}
\end{pps}
\begin{proof}
Since $\lri= \bigcup_{m=1}^{q}(\pi_m+\p)$, for some representatives $\pi_m$ of the classes of $\lri/\p$, there exist $k\in \N$ and representatives $A_1, \dots, A_k$ of 
$\lri^{2n}/\p^{(2n)}$ such that 
\[\lri^{2n}=\left(\cup_{m=1}^{k}A_m+\Mat_{2\times n}(\p)\right)\cup\Mat_{2\times n}(\p),\]
where $\Mat_{2\times n}(\p)$ denotes the set of all $2\times n$-matrices over $\p$. Hence \[\W_{2n}=\cup_{m=1}^{k}A_m+\Mat_{2\times n}(\p).\]
In the following, we evaluate the integrals
\[\mathcal{I}_{A_m}(s,r):=\int_{(y,\underline{x}) \in \p\times A_m+\Mat_{2\times n}(\p)}|y|_{\p}^{r}\|\{M_{ij}(\underline{x})\mid 1\leq i<j\leq n\}\cup\{ y\}\|_{\p}^{s}d\mu.\]
If $\underline{x}\in A_m+\Mat_{2\times n}(\p)$, then $\underline{x}$ and $A_m$ have the same number of invertible elementary divisors modulo~$\p$. Let us then consider separately the case that $A_m$ has exactly one invertible elementary divisor and the case that $A_m$ has two invertible elementary divisors. 
For simplicity, assume that $A_1, \dots, A_t$ have exactly one invertible elementary divisor and that $A_{t+1}, \dots, A_k$ have two invertible elementary divisors for some $t \in [k]_0$.

\noindent
\textbf{Case 1}: Suppose that $m \in [t]$, that is, $A_m$ has exactly one elementary divisor modulo~$\p$. 
Then, each $\underline{x}=(x_{ij})\in A_m+\Mat_{2\times n}(\p)$ has exactly one elementary divisor modulo~$\p$. In particular,  
$v_{\p}(M_{ij}(\underline{x}))\geq 1$ for all $1\leq i< j\leq n$. 
By making a suitable change of variables, we can consider $A_m$ to be the matrix with $(1,1)$-entry $1$ and $0$ elsewhere.
Consequently, $M_{ij}(\underline{x})$ is as in Lemma~\ref{Lemma:Minors}, so that $\|\{M_{ij}(\underline{x})\mid{1\leq i < j \leq n}\}\|_{\p}=\|M_{12}(\underline{x}), \dots, M_{1n}(\underline{x})\|_{\p}$.
Therefore
\begin{align*}\mathcal{I}_{A_m}(s,r) 
 &=\int_{(y,\underline{x})\in \p\times \Mat_{2\times n}(\p)}|y|_{\p}^{r}\|M_{12}(\underline{x}),\dots, M_{1n}(\underline{x}), y\|_{\p}^{s}d\mu\\
 &=\mu(\p^{n+1})\int_{(y,x_1,\dots,x_{n-1}) \in \p\times\p^{n-1}}|y|_{\p}^{r}\|x_1,\dots, x_{n-1}, y\|_{\p}^{s}d\mu\\
 &=q^{-n-1}\frac{(1-q^{-1})(1-q^{-r-n})q^{-r-s-n}}{(1-q^{-r-s-n})(1-q^{-r-1})},
\end{align*}
where the domain of integration of the integral in the second equality is justified by the translation invariance of the Haar measure and the last equality is due to Proposition~\ref{integral2}.

\noindent
\textbf{Case 2}: We now assume that $m \in \{t+1, \dots, k\}$, that is, $A_m$ has two invertible elementary divisors modulo~$\p$. In this case, each $\underline{x} \in A_m+\Mat_{2\times n}(\p)$ has two elementary divisors modulo $\p$, which 
means that at least one of the $M_{ij}(\underline{x})$ has valuation zero. Consequently, 
\[\mathcal{I}_{A_m}(s,r)=\int_{(y,\underline{x})\in \p\times\p^{2n}}|y|_{\p}^{r}d\mu
=\frac{q^{-2n-r-1}(1-q^{-1})}{1-q^{-r-1}}.\]

We conclude the proof using the fact that there are $(q+1)(q^n-1)$ matrices of rank $1$ and $q(q^n-1)(q^{n-1}-1)$ matrices of rank $2$ in $\Mat_{2\times n}(\F_{q})$ and, consequently, 
\begin{align*}
 &\int_{(y,\underline{x})\in\p\times\W_{2n}}|y|_{\p}^{r}\|\{M_{ij}(\underline{x})\mid 1\leq i<j\leq n\}\cup{y}\|_{\p}^{s}d\mu
 =\sum_{m=1}^{k}\mathcal{I}_{A_m}(s,r)\\
 &=(q+1)(q^n-1)\mathcal{I}_{A_t}(s,r)+q(q^n-1)(q^{n-1}-1)\mathcal{I}_{A_k}(s,r)\\
 &=\frac{(q^n-1)(1-q^{-1})q^{-r-2n-1}}{(1-q^{-1-r})(1-q^{-r-s-n})}\left((q+1)(1-q^{-r-n})q^{-s}+(q^{n}-q)(1-q^{-r-s-n})\right),
\end{align*}
as desired. 
\qedhere
\end{proof}

\subsection{Bivariate zeta functions in terms of \texorpdfstring{$\p$}{p}-adic integrals}\label{ppspadic}
We now recall from~\cite[Proposition~4.8]{PL18} how to write local bivariate zeta functions in terms of $\p$-adic integrals.
Let $\g=\Lambda(\lri)=\Lambda\otimes_{\ri} \lri$. Let also $\g'$ be the derived Lie sublattice of $\g$, and let $\z$ be its centre. Set 
\begin{align*} & h=\rk(\g), & & a=\rk(\g/\z), & & b=\rk(\g'), & & r= \rk(\g/\g'), & & z=\rk(\z).\end{align*} 
For $\Lambda \in \{\Fnd,\Gn,\Hn\}$, these numbers are given in Table~\ref{table}.
\begin{table*}[h] \caption {Ranks related to groups of type $F$, $G$ and $H$.}\label{table} 
	\centering
		\begin{tabular}{ c | c | c | c}
			$\Lambda$ & $h=\rk(g)$ & $a=\rk(\g/\z)$ & $b=\rk(\g')=\rk(\z)=z$  \\[0.5ex]
				\hline 
				$\Fnd$ & $\binom{2n+\delta+1}{2}$ & $2n+\delta$ & $\binom{2n+\delta}{2}$ \Tstrut\Bstrut\\[0.1ex]
				
				$\Gn$ & $n^2+2n$ & $2n$ & $n^2$ \Tstrut\Bstrut\\[0.1ex]
				
				$\Hn$ & $\binom{n+1}{2}+2n$ & $2n$ & $\binom{n+1}{2}$ \\
		\end{tabular}
\end{table*}

We now recall the bases $\mathbf{e}$ and $\mathbf{f}$ of~\cite[Section~4.1]{PL18} in the context of the Lie lattice $\Lambda \in \{\Fnd,\Gn,\Hn\}$. Set $\mathbf{e}=(x_1, \dots, x_a)$, where the $x_i$ are the elements appearing in the presentation of~$\Lambda$ of Definition~\ref{fgh}. Let $\overline{\phantom{I}}$ denote the natural surjection $\g \to \g/\z$. Then $\overline{\mathbf{e}}:=(\overline{x_1}, \dots, \overline{x_a})$ is an $\lri$-basis of $\g/\z$. 

We define the sets 
\[\DL=\begin{cases}
                \{(i,j) \in [2n+\delta]^2\mid 1\leq i<j \leq 2n+\delta\},&\text{ if }\Lambda=\Fnd,\\
                [n]^2,&\text{ if }\Lambda=\Gn,\\
                \{(i,j) \in [n]^2\mid 1\leq i\leq j \leq n\},&\text{ if }\Lambda=\Hn.
               \end{cases}\]
Let $y_{ij}$ be the elements appearing in the relations of the presentation of~$\Lambda$ of Definition~\ref{fgh}. We order the set $\mathbf{f}=\{y_{ij}\}_{(i,j) \in \DL}$, by setting $y_{ij}>y_{kl}$, whenever either $i <k$ or $i=k$ and $j<l$.
We then write $\mathbf{f}=(y_{ij})_{(i,j)\in \DL}=(f_1,\dots, f_b)$ so that $f_1> \dots > f_b$. This ordering is fixed so that we have a unique map $\omega: \DL \to [b]$ satisfying $[x_i,x_j]=f_{\omega(i,j)}$, fact that is used in Section~\ref{scommutator}.

For $i,j \in [a]$ and $k \in [b]$, let $\lambda_{ij}^{k}\in \lri$ be the structure constants satisfying 
\[[x_i,x_j]=\sum_{k=1}^{b}\lambda_{ij}^{k}f_k.\]
\begin{dfn}\label{commutator}\cite[Definition~2.1]{ObVo15} The $A$\emph{-commutator} and the $B$\emph{-commutator matrices} of $\g$ with respect to $\mathbf{e}$ and $\mathbf{f}$ are, respectively,
\begin{align*}
 A_{\Lambda}(X_1, \dots, X_a)&=\left( \sum_{j=1}^{a}\lambda_{ij}^{k}X_j\right)_{ik}\in \Mat_{a\times b}(\lri[\underline{X}]),\text{ and} \\
 B_{\Lambda}(Y_1,\dots, Y_b)&=\left( \sum_{k=1}^{b}\lambda_{ij}^{k}Y_k\right)_{ij}\in \Mat_{a\times a}(\lri[\underline{Y}]),
\end{align*}
where $\underline{X}=(X_1, \dots, X_a)$ and $\underline{Y}=(Y_1, \dots, Y_b)$ are independent variables. 
\end{dfn}
The $B$-commutator matrices of $\Lambda$ are the following:
\begin{itemize}
 \item $B_{\Fnd}(\underline{Y})$ is the generic antisymmetric $(2n+\delta)\times (2n+\delta)$-matrix in the variables $\underline{Y}=(Y_1, \dots, Y_b)$,
 \item $B_{\Gn}(\underline{Y})= \left[ \begin {array}{cc} 0&M(\underline{Y})\\ \noalign{\medskip}-M(\underline{Y})^{\textup{tr}}&0\end {array}\right]$, where $M(\underline{Y})$ is
 the generic $n\times n$-matrix in the variables $\underline{Y}=(Y_1, \dots, Y_b)$ and $M(\underline{Y})^{\textup{tr}}$ is its transpose,
 \item $B_{\Hn}(\underline{Y})= \left[ \begin {array}{cc} 0&S(\underline{Y})\\ \noalign{\medskip}-S(\underline{Y})&0\end {array}\right]$, where $S(\underline{Y})$ is the 
 generic symmetric $n\times n$-matrix in the variables $\underline{Y}=(Y_1, \dots, Y_b)$.
\end{itemize}
The precise form of the $A$-commutator matrix of each $\Lambda \in \{\Fnd,\Gn,\Hn\}$ is given in Section~\ref{scommutator}.

We now recall how to express the bivariate zeta functions of Definition~\ref{biva} in terms of $\p$-adic integrals.
Let $n \in \N$ and let $\mathcal{R}(\underline{X})=\mathcal{R}(X_1, \dots, X_n)$ be a matrix of polynomials $\mathcal{R}(\underline{X})_{ij}\in\lri[\underline{X}]$. In the following, we express bivariate zeta functions in terms of the $\p$-adic integrals
\[\pint_{\mathcal{R}}(\underline{r},t)=\frac{1}{1-q^{-1}}\int_{(w,\underline{x}) \in \p\times \W_{n}}|w|_{\p}^{t}\prod_{k=1}^{u}
\frac{\|F_{k}(\mathcal{R}(\underline{x})) \cup wF_{k-1}(\mathcal{R}(\underline{x}))\|_{\p}^{r_k}}{\|F_{k-1}(\mathcal{R}(\underline{x}))\|_{\p}^{r_{k}}}d\mu,
\]
where $\mu$ is the additive Haar measure normalised so that $\mu(\lri^{n+1})=1$, $u \in \N$, and $F_j(\mathcal{R}(\underline{x}))$ is the set of nonzero $j \times j$-minors of $\mathcal{R}(\underline{x})$.

\begin{pps}\label{intpadic}\cite[Proposition~4.8]{PL18} The bivariate zeta functions of $\G(\lri)$ can be expressed as
\begin{align}
\nonumber \rlvlzf_{\G(\lri)}(s_1,s_2)&=\frac{1}{1-q^{r-s_2}}\left(1+\pint_{B_\Lambda}\left(\tfrac{2-s_1}{2}\mathbf{1}_{\ub},\ub s_1+s_2+2\ub-h-1\right)\right),\\
\label{ccpadic} \clvlzf_{\G(\lri)}(s_1,s_2)&=\frac{1}{1-q^{z-s_2}}\left(1+\pint_{A_\Lambda}\left((1-s_1)\mathbf{1}_{\ua},\ua s_1+s_2+\ua-h-1\right)\right),
\end{align}
where, for $u \in \N$, we denote $\mathbf{1}_u=(1, \dots, 1) \in \Z^u$ and  
\begin{align*}2\ub&=\max\{\rk_{\Frac(\lri)}B_\Lambda(\mathbf{z}) \mid \mathbf{z}\in \lri^b\},\\
\ua&=\max\{\rk_{\Frac(\lri)}A_\Lambda(\mathbf{z}) \mid \mathbf{z}\in \lri^a\}.
\end{align*}
\end{pps}
In particular, $\ub=n$, for all $\Lambda \in \{\Fnd,\Gn,\Hn\}$. 

Via specialisation~\eqref{specialisation}, we obtain:
\begin{align}\nonumber\kzeta_{\G(\lri)}(s)&=
\frac{1}{1-q^{z-s}}\left(1+\pint_{A_\Lambda}\left(-\mathbf{1}_{\ua},s+\ua-h-1\right)\right)\\
&\label{rtok}=\frac{1}{1-q^{r-s}}\left(1+\pint_{B_\Lambda}\left(-\mathbf{1}_{\ub},s+2\ub-h-1\right)\right).
\end{align}

\begin{rmk}\label{sod}
Formula~\eqref{kzetaF} for the class number zeta function of $F_{n,\delta}(\lri)$ is also a consequence of~\cite[Proposition~5.11]{ro17}. 
This proposition gives a formula for the \emph{ask zeta function} $\textup{Z}_{\mathfrak{so}_{d}(\lri)}^{\textup{ask}}(T)$ of the orthogonal Lie algebra $\mathfrak{so}_d(\lri)$, $d \in \N$; see definition in~\cite[Definition~1.3]{ro17}. 
Since $\{B_{\Fnd}(\underline{x}) \mid \underline{x} \in \lri^b\}=\mathfrak{so}_a(\lri)$, when comparing the $p$-adic integral~\cite[(4.3)]{ro17} with~\eqref{rtok}, we see that 
\[\kzeta_{F_{n,\delta}(\lri)}(s)=\textup{Z}_{\mathfrak{so}_a(\lri)}^{\textup{ask}}(q^{-s+\binom{a}{2}}),\] and hence 
\cite[Proposition~5.11]{ro17} shows~\eqref{kzetaF}.
\end{rmk}

\section{Bivariate conjugacy class zeta functions---proof of Theorem~\ref{cclvl}}\label{pcclvl}
The goal of this section is to prove Theorem~\ref{cclvl}. This is done by calculating explicitly the corresponding integrals~\eqref{ccpadic} of Proposition~\ref{intpadic} for each type of $\G\in\{F_{n,\delta},G_n,H_n\}$. For this purpose, in Section~\ref{scommutator}, we describe the $A$-commutator matrices for each type of $\G$ and then, in Sections~\ref{typeF}, \ref{typeG}, and~\ref{typeH}, we use the obtained descriptions to determine the terms 
\begin{equation}\label{terms}\prod_{k=1}^{\ua}
\frac{\|F_{k}(\commA(\underline{x})) \cup wF_{k-1}(\commA(\underline{x}))\|_{\p}^{-1-s_1}}{\|F_{k-1}(\commA(\underline{x}))\|_{\p}^{-1-s_1}}
\end{equation}
appearing in the integrands of the $\p$-adic integrals~\eqref{ccpadic}, and then we evaluate the corresponding integral.

\subsection{\texorpdfstring{$A$}{A}-commutator matrices}\label{scommutator}
In this section, we describe the $A$-commutator matrices of each type of~$\G$. 
Recall that the $A$-commutator matrix of $\g=\Lambda(\lri)$ is given with respect to the bases $\mathbf{e}=(x_1, \dots, x_a)$ and $\mathbf{f}=(y_{ij})_{(i,j) \in \DL}$ of Section~\ref{ppspadic}. For simplicity, we would like to write the ordered basis~$\mathbf{f}$ in the form $\mathbf{f}=(f_1, \dots, f_b)$. The following lemma gives a way to determine which $y_{ij}$ corresponds to each $f_k$.

\begin{lem}Let $\wL: \DL\to [b]$ denote the map satisfying $y_{ij}=f_{\wL(i,j)}$. Then 
\[\wL(i,j)=\begin{cases}
               (i-1)a-\binom{i+1}{2}+j,&\text{ if }\Lambda=\Fnd,\\
               (i-1)n+j,&\text{ if }\Lambda=\Gn,\\
               (i-1)n-\binom{i}{2}+j,&\text{ if }\Lambda=\Hn.
              \end{cases}\] 
\end{lem}
\begin{proof}
Let us show the identity for $\Lambda=\Fnd$. Recall that the ordering on the $y_{ij}$ is given by $y_{ij}>y_{kl}$ whenever $i<k$ or $i=k$ and $j<l$. In particular, we have that 
\[ 	y_{12}>y_{13}> \dots > y_{1n}. \]
In other words, $y_{12}$ corresponds to $f_1$, that is $\omega_{\Fnd}(1,2)=1$, and $y_{1j}$ corresponds to $f_{j-1}$:
\[\omega_{\Fnd}(1,j)=j-1,\text{ for all }j \in \{2, \dots, a\}.\]

Let us now show the induction step. Suppose that for some $1 \leq i < j \leq a$ it holds that $\omega_{\Fnd}(i,j)=(i-1)a-\binom{i+1}{2}+j$.
Then, if $j \neq a$, we have
\[\omega_{\Fnd}(i,j+1)=\omega_{\Fnd}(i,j)+1=(i-1)a-\tbinom{i+1}{2}+j+1, \]
as desired. We now assume $j=a$. By definition of $>$, the element $y_{(i+1)(i+2)}$ is the successor of $y_{ia}$, for each $i\in[n-2]$.
In other words,
\[\omega_{\Fnd}(i+1,i+2)=\omega_{\Fnd}(i,a)+1.\]
Thus
\begin{align*}\omega_{\Fnd}(i+1,i+2)&=\omega_{\Fnd}(i,a)+1=(i-1)a-\tbinom{i+1}{2}+a+1 \\
&= ia-\left(\tbinom{i+2}{2}-(i+1)\right)+1\\
&= ia - \tbinom{i+2}{2}+i+2.	
\end{align*}

The other cases follow from similar arguments.
\end{proof}

We now divide the $A$-commutator matrix of~$\g$ into submatrices and describe the entries of these submatrices. 
Denote by $\commGenm$ the submatrix of~$\commGen$ composed of all columns of~$\commGen$ with index in 
\begin{align*}\CLm&=\{\wL(m,j)\mid (m,j)\in \DL\} \\
																			&=\begin{cases}
																			\{\omega_{\Fnd}(m,j)\mid j \in \{m+1, m+2, \dots, a\}\},&\text{ if }\Lambda=\Fnd,\\
																			\{\omega_{\Gn}(m,j) \mid j \in [n]\},& \text{ if }\Lambda=\Gn,\\
																			\{\omega_{\Hn}(m,j) \mid j \in \{m, m+1, \dots, n\}\}, & \text{ if }\Lambda=\Hn.
																			\end{cases}
\end{align*}

Let $\underline{X}=(X_1, \dots, X_a)$ be a vector of variables, we have that 
\begin{align*}
 \commF&=\left[ \begin {array}{cccc} A_{\Fnd}^{(1)}(\underline{X})&A_{\Fnd}^{(2)}(\underline{X})&\dots&A_{\Fnd}^{(a-1)}(\underline{X}) \end {array} \right],\nonumber\\
 \commGen&=\left[ \begin {array}{cccc} A_{\Lambda}^{(1)}(\underline{X})&A_{\Lambda}^{(2)}(\underline{X})&\dots&A_{\Lambda}^{(n)}(\underline{X}) \end {array} \right], \text{ for }\Lambda\in \{\Gn,\Hn\}.
\end{align*}

We define numbers~$\vLm$ so that we can simply write \[\CLm=\{\vLm+1, \vLm+2, \dots, \vLm +\kLm\},\] for some $\kLm \in \N$. 
This way, the $j$th column of $\commGenm$ is the $(\vLm+j)$th column of $\commGen$.
These numbers are defined as follows. 
\begin{align*}\vLm &=\begin{cases}
							(m-1)a-\binom{m+1}{2}+m,&\text{ if }\Lambda=\Fnd\\
               (m-1)n,&\text{ if }\Lambda=\Gn,\\
							 (m-1)n-\binom{m}{2}+m-1,&\text{ if }\Lambda=\Hn, \text{ and }
              \end{cases}\\
\kLm&=\begin{cases}
     a-m,&\text{ if }\Lambda=\Fnd,\\
     n,&\text{ if }\Lambda=\Gn,\\
     n-m+1,&\text{ if }\Lambda=\Hn.
 \end{cases}
\end{align*}

In the following, we describe separately the submatrices~$\commGenm$ of~$\commGen$ for each $\Lambda \in \{\Fnd,\Gn, \Hn\}$. 
For simplicity, in the next sections we write~$\omega$ and~$\vm$ for $\omega_{\Lambda}$ and $\vLm$, respectively. 

\subsubsection{\texorpdfstring{$A$}{A}-commutator matrices of groups of type \texorpdfstring{$F$}{F}}
In this section, we fix $\Lambda=\Fnd$ and determine the entries of~$\commFm$.

Fix $m \in [a-1]$. Recall that the $(i,l)$-entry of $\commGenm$ is the $(i,\vm+l)$-entry of $\commGen$, i.e.\ for $i \in [a]$ and $l \in [\kFm]$, we have that
 \[\commGenm_{il}=\commGen_{i(\vm+l)}=\sum_{j=1}^{a}\lambda_{ij}^{\vm+l}X_j.\]
Thus, we need to calculate the values of $\lambda_{ij}^{\vm+l}$ in the sum above in order to determine the entries of $\commGenm$.
Recall that the $\lambda_{ij}^{k}$ are defined by 
\[ [e_i, e_j] = \sum_{k=1}^{b} \lambda_{ij}^{k}f_{k},\] 
and that we have defined $\omega(i,j)$ so that $f_{\omega(i,j)}=y_{ij}$.

For $\Lambda=\Fnd$, the sturcture constants~$\lambda_{ij}^k$ are determined by the relations
\[[x_i,x_j]=y_{ij}, \text{ for all }1 \leq i < j \leq a.\] 
Hence $\lambda_{ij}^{k}=1$ if and only if $k=\omega(i,j)$. 

In the following, we fix $i \in [a]$, $l \in [\kFm]$ and let~$j$ run in~$[a]$.

\noindent
\textbf{Case $\mathbf{i=m}$.} In this case we have
		\[ \lambda_{ij}^{\vm+l}=	\begin{cases}
																1, &\text{ if }\vm+l=\omega(m,j)\\
																0, &\text{ otherwise.}
														\end{cases}\]
		Since $\vm+l=\omega(m,j)$ if and only if $j=m+l$, it follows that 
		\[\commFm_{ml}=X_{m+l}.\]

\noindent		
\textbf{Case $\mathbf{i=m+l}$.} Observe that, if $j\geq i$, then $\lambda_{ij}^{\vm+l}$ can only be nonzero if $\vm+l=\omega(i,j)$. However, this equality is only possible for $i=m$. Consequently, 
		\[\lambda_{ij}^{\vm+l}=-\lambda_{ji}^{\vm+l}=	\begin{cases}
																										-1, &\text{ if }j<i\text{ and }\vm+l=\omega(j,i),\\
																										0, &\text{ otherwise}.
																									\end{cases}\]
	Since $\vm+l=\omega(j,i)$ if and only if $j=m$, it follows that 
	\[\commFm_{(m+l)l}=-X_m.\]

\noindent	
\textbf{Case $\mathbf{i\neq m}$ and $\mathbf{i \neq m+l}$.} In this case, we find that $\lambda_{ij}^{\vm+l}=0$ for all $j \in [a]$. Hence
	\[\commFm_{il}=0.\]

Given $s,r \in \N$, let $\textbf{0}_{s\times r}$ denote the $s\times r$-zero matrix and let $\textbf{1}_{s}$ denote the $s\times s$-identity matrix, both over 
$\lri[\underline{X}]$. It follows that, for each $m\in[a-1]$, 
\begin{align}\label{commF}
 \commFm & = \begin{bmatrix}
 \multicolumn{4}{c}{\multirow{2}{*}{{\Large \textbf{0}}$_{(m-1)\times (2n+\delta-m)}$}}\\
  & & & \\
     X_{m+1} & X_{m+2} & \dots & X_{2n+\delta}  \\
\multicolumn{4}{c}{\multirow{2}{*}{$-X_m$ {\Large \textbf{1}}$_{(2n+\delta-m)}$}} \\
& & & 
\end{bmatrix}
\in \Mat_{(2n+\delta)\times (2n+\delta-m)}(\lri[\underline{X}]).
\end{align}

\subsubsection{\texorpdfstring{$A$}{A}-commutator matrices of groups of type \texorpdfstring{$G$}{G}}
In this section, we fix $\Lambda=\Gn$.
 
For each $m \in [n]$, we determine the entries of $\commGm$, which are of the form 
\[\commGm_{il}= \sum_{j=1}^{2n}\lambda_{ij}^{\vm+l}X_j,\]
for $i \in [2n]$ and $l \in [\kGm]$.
It suffices to determine the values of $\lambda_{ij}^{\vm+l}$ in the sum above. 
For $\Lambda=\Gn$, the structure constants~$\lambda_{ij}^{k}$ are determined by the relations
\[ [x_i, x_{n+j}]=y_{ij}=f_{\omega(i,j)}, ~1 \leq i, j \leq n. \]
Thus, given $i, j \in [n]$ and $k \in [\kGm]$, we have that 
\begin{align*} \lambda_{ij}^{k}&=\lambda_{(n+i)(n+j)}^{k}=0, \\ 
								\lambda_{i(n+j)}^{k}&=1 \text{ if and only if }\omega(i,j)=k.
\end{align*}
In the following, we fix $i \in [2n]$, $l \in [\kGm]$ and let~$j$ run in~$[n]$.

\noindent
\textbf{Case $\mathbf{i=m}$.} In this case we have
		\[ \lambda_{i(n+j)}^{\vm+l}=	\begin{cases}
																1, &\text{ if }\vm+l=\omega(m,j)\\
																0, &\text{ otherwise.}
														\end{cases}\]
		Since $\vm+l=\omega(m,j)$ if and only if $j=l$, it follows that 
		\[\commGm_{ml}=\sum_{j=1}^{n}\lambda_{i(n+j)}^{\vm+l}X_{n+j}=X_{n+l}.\]

\noindent
\textbf{Case $\mathbf{i=n+l}$.} In this case
		\[\lambda_{ij}^{\vm+l}=-\lambda_{j(n+l)}^{\vm+l}=	\begin{cases}
														-1, &\text{ if }\vm+l=\omega(j,l),\\
														0, &\text{ otherwise}.
		\end{cases}\]

	As $\vm+l=\omega(j,l)$ if and only if $j=m$, it follows that 
	\[\commGm_{(n+l)l}=-X_m.\]
	
\noindent
\textbf{Case $\mathbf{i\in [2n]\setminus \{m,n+l\}}$.} In this case, we find that $\lambda_{ij}^{\vm+l}=0$ for all $j \in [2n]$. Hence
	\[\commGm_{il}=0.\]

Therefore, for each $m \in [n]$,
\begin{align}
\label{commG}\commGm&=\begin{bmatrix}
			\multicolumn{4}{c}{\multirow{2}{*}{{\Large \textbf{0}}$_{(m-1)\times n}$}} \\
			&  &  &  \\
			X_{n+1} & X_{n+2} & \dots & X_{2n}  \\
			\multicolumn{4}{c}{\multirow{2}{*}{{\Large \textbf{0}}$_{(n-m)\times n}$}}\\
			&&&\\
			\multicolumn{4}{c}{\multirow{2}{*}{$-X_m${\Large \textbf{1}}$_{n}$}}\\
			&&&
\end{bmatrix}
\in \Mat_{2n\times n}(\lri[\underline{X}]).
 \end{align}

\subsubsection{\texorpdfstring{$A$}{A}-commutator matrices of groups of type \texorpdfstring{$H$}{H}}
In this section, we fix $\Lambda=\Hn$.
 
For each $m \in [n]$, we determine the entries of $\commHm$, which are of the form 
\[\commHm_{il}= \sum_{j=1}^{2n}\lambda_{ij}^{\vm+l}X_j,\]
for $i \in [2n]$ and $l \in [\kHm]$.
Let us determine the values of $\lambda_{ij}^{\vm+l}$ in the sum above. 
For $\Lambda=\Hn$, the structure constants~$\lambda_{ij}^{k}$ are determined by the relations
\[ [x_i, x_{n+j}]=[x_j,x_{n+i}]=y_{ij}=f_{\omega(i,j)}, ~1 \leq i \leq j \leq n. \]
Thus, given $i, j \in [n]$ and $k \in [\kHm]$, we have that 
\begin{align*}\lambda_{ij}^{k}&			=\lambda_{(n+i)(n+j)}^{k}=0, \\ 
							\lambda_{i(n+j)}^{k}&	=\lambda_{j(n+i)}^{k}\text{ are nonzero if and only if }\omega(i,j)=k.
\end{align*}

In the following, we fix $i \in [2n]$, $l \in [\kHm]$, and let~$j$ run in~$[n]$.

\noindent
\textbf{Case $\mathbf{i=m}$.} 
				Observe that if $j<m$, then $\lambda_{i(n+j)}^{\vm+l}=0$ because $\omega(j,m)=\vm+l$ is only possible if $j=m$. 
				Consequently,
				\[ \lambda_{i(n+j)}^{\vm+l}=\begin{cases}
																1, &\text{ if }j \geq m\text{ and }\vm+l=\omega(m,j),\\
																0, &\text{ otherwise.}
														\end{cases}\]
				As $\vm+l=\omega(m,j)$ if and only if $j=m+l-1$, it follows that 
				\[\commHm_{ml}=X_{n+m+l-1}.\]	

\noindent
\textbf{Case $\mathbf{i=m+l-1}$ with $\mathbf{l \neq 1}$.} 
		Observe that if $j >i$, then $\lambda_{i(n+j)}^{\vm+l}=0$ because $\omega(i,j)=\vm+l$ would only be possible if $i=m$. Consequently, 
				\[ \lambda_{i(n+j)}^{\vm+l}=	\begin{cases}
																1, &\text{ if }j \leq i \text{ and }\vm+l=\omega(j,i),\\
																0, &\text{ otherwise.}
														\end{cases}\]
				As $\vm+l=\omega(j,m+l-1)$ if and only if $j=m$, it follows that 
				\[\commHm_{(m+l-1)l}=	X_{n+m}.\]

\noindent
		\textbf{ Case $\mathbf{i=n+m}$.} Similarly to the case $i=m$, if $j<m$, then $\lambda_{(n+m)j}^{\vm+l}=0$.
				Consequently,
					\[\lambda_{ij}^{\vm+l}=-\lambda_{j(n+m)}^{\vm+l}=	\begin{cases}
														-1, &\text{ if }j \geq m\text{ and }\vm+l=\omega(m,j),\\
														0, &\text{ otherwise}.
														\end{cases}\]
				As $\vm+l=\omega(m,j)$ if and only if $j=m+l-1$, it follows that 
				\[\commHm_{(n+m)l}=-X_{m+l-1}.\]

\noindent
\textbf{Case $\mathbf{i = n+m+l-1}$ with $\mathbf{l \neq 1}$.} 
				Similarly to the case $i=m+l-1$, if $j > m+l-1 >m$, then $\lambda_{(n+m+l-1)j}^{\vm+l}=0$. Consequently,
				\[ \lambda_{ij}^{\vm+l}=-\lambda_{j(n+m+l-1)}^{\vm+l}	=
														\begin{cases}
																-1, &\text{ if }j \leq m+l-1 \text{ and }\vm+l=\omega(j,m+l-1),\\
																0, &\text{ otherwise.}
														\end{cases}\]
				As $\vm+l=\omega(j,m+l-1)$ if and only if $j=m$, it follows that 
				\[\commHm_{(n+m+l-1)l}=-X_m.\]

\noindent
\textbf{Case $\mathbf{i\in \{1, \dots, m-1\}\cup \{n+1, \dots, n+m-1\}}$.} In this case, we find that $\lambda_{ij}^{\vm+l}=0$ for all $j \in [2n]$. Hence
				\[\commHm_{il}=0.\]

Therefore, for each $m \in [n]$,
\begin{align}
\label{commH}\commHm&=\begin{bmatrix}
			  \multicolumn{4}{c}{\multirow{2}{*}{{\Large \textbf{0}}$_{(m-1)\times (n-m+1)}$}}\\
			  &  &  &  \\
			X_{n+m} & X_{n+m+1} & \dots & X_{2n}  \\
			& X_{n+m} &  &   \\
			&  & \ddots &   \\
			&  &  & X_{n+m}  \\
			\multicolumn{4}{c}{\multirow{2}{*}{{\Large \textbf{0}}$_{(m-1)\times (n-m+1)}$}}\\
			&  &  &  \\
			-X_{m} & -X_{m+1} & \dots & -X_{n}  \\
			& -X_{m} &  &   \\
			&  & \ddots &   \\
			&  &  & -X_{m}  \\
\end{bmatrix}
\in \Mat_{2n\times (n-m+1)}(\lri[\underline{X}]).
\end{align}
\begin{exm}
 The following examples illustrate the form of the commutator matrix for each type of group scheme. 
\begin{align*} A_{\mathcal{F}_{2,0}}(\underline{X})&=\left[ \begin {array}{cccccc}  X_2 & X_3 & X_4 & & & 
\\ \noalign{\medskip}-X_1 & & & X_3 & X_4 & 
\\ \noalign{\medskip} & -X_1 & & -X_2 & & X_4
\\ \noalign{\medskip} & & -X_1 & &-X_2 & -X_3
\end {array} \right],\\
A_{\mathcal{G}_{3}}(\underline{X})&=\left[ \begin {array}{ccccccccc}  X_4 & X_5 & X_6 & & & & & & 
\\ \noalign{\medskip} & & & X_4 & X_5 & X_6 & & & 
\\ \noalign{\medskip} & & & & & & X_4 & X_5 & X_6
\\ \noalign{\medskip}-X_1 & & & -X_2 & & & -X_3 & & 
\\ \noalign{\medskip} & -X_1 & & & -X_2 & & & -X_3 &
\\ \noalign{\medskip} & & -X_1 & & & -X_2 & & & -X_3
 \end {array} \right],\\
 \end{align*}
 \begin{align*}
A_{\mathcal{H}_{3}}(\underline{X})&=\left[ \begin {array}{cccccc}  
		      X_4 & X_5 & X_6 & & & 
\\ \noalign{\medskip} & X_4 & & X_5 & X_6 & 
\\ \noalign{\medskip} & & X_4 & & X_5 & X_6
\\ \noalign{\medskip}-X_1 & -X_2 & -X_3 & & & 
\\ \noalign{\medskip} & -X_1 & & -X_2 & -X_3 & 
\\ \noalign{\medskip} & & -X_1 & & -X_2 & -X_3
\end {array} \right],  
\end{align*}
where the omitted entries equal zero.
\end{exm}
It is not difficult to see that $\commGen$ has rank $a-1$ in all cases, that is, $\ua=a-1$. 

We now proceed to a detailed analysis of the $A$-commutator matrix in each individual type.
 
\subsection{Conjugacy class zeta functions of groups of type \texorpdfstring{$F$}{F}}\label{typeF}
In this section, we prove Theorem~\ref{cclvl} for groups of type~$F$. 
More precisely, we use the descriptions~\eqref{commF} of the submatrices~$\commFm$ in order to determine the integrands of~\eqref{terms}. After that, we evaluate the corresponding integral.

\begin{lem}\label{simpleF} For $w \in \p$ and $\mathbf{x}\in \W_{a}$, that is, for $\mathbf{x}\in \lri^a$ such that $v_{\p}(\mathbf{x})=0$,
\begin{equation}\label{integrandF}\frac{\|F_{k}(A_{\mathcal{F}_{n,\delta}}(\mathbf{x}))\cup wF_{k-1}(A_{\mathcal{F}_{n,\delta}}(\mathbf{x}))\|_{\p}}
{\|F_{k-1}(A_{\mathcal{F}_{n,\delta}}(\mathbf{x}))\|_{\p}}=1, 
\text{ for all }k\in[a-1].\end{equation}
\end{lem}

\begin{proof}
Expression~\eqref{commF} shows that the columns of $\commF$ are of the form 
\begin{align}\label{columnf}
 \begin{matrix}
\vphantom{a}\\
    \coolleftbrace{\phantom{j}i\text{th row}}{X_{n+j}}\\
    \vphantom{a}\\
    \vphantom{b}\\
    \coolleftbrace{j\text{th row}\phantom{i}}{i}\\
    \vphantom{a}
\end{matrix}
\begin{bmatrix}
       \\
      X_{j} \\
       \\
       \\
      -X_i  \\
      \\ 
\end{bmatrix},
&\text{ for }1 \leq i < j \leq a,
\end{align}
where the nondisplayed entries equal $0$. Denote column~\eqref{columnf} by $c_{ij}$, that is, $c_{ij}$ is the unique column of $\commFm$ with $X_j$ in the $i$th row, with $-X_i$ in the $j$th row and zero elsewhere.

For each $m \in [a]$, consider the $a\times (a-1)$-submatrix $S_m(\underline{X})$ of~$\commF$ composed of columns $c_{1m}$, $c_{2m}$, $\dots$, $c_{(m-1)m}$, $c_{m(m+1)}$, $c_{m(m+2)}$, $\dots$, $c_{ma}$, in this order. Then $S_m(\underline{X})$ is the matrix with main diagonal given by $X_m$ in the first $m-1$ entries and $-X_m$ in the remaining main diagonal entries, and the other nonzero entries are the ones of row $m$, which is given by
\[(-X_1, \dots, -X_{m-1}, X_{m+1}, \dots, X_{a}).\]

Given $\mathbf{x}\in\W_{a}$, it is clear that, for at least one $m_0 \in [a]$, the matrix $S_{m_0}(\mathbf{x})$ has rank $a-1$.
That is, for each $k \in [a-1]$, there exists a $k\times k$-minor of $S_{m_0}(\mathbf{x})$ which is a unit.
Since the $k\times k$-minors of $S_{m_0}(\mathbf{x})$ are elements of $F_{k}(A_{\Fnd}(\mathbf{x}))$, it follows that~\eqref{integrandF} holds.
\qedhere
\end{proof}

Proposition~\ref{intpadic}, Lemma~\ref{simpleF} and Proposition~\ref{integral1} yield
\begin{align*}&\clvlzf_{\mathcal{F}_{n,\delta}(\lri)}(s_1,s_2)\\
&=\frac{1}{1-q^{\binom{2n+\delta}{2}-s_2}}\left(1+\frac{1}{1-q^{-1}}\int_{(w,\underline{x})\in\p\times\W_{2n+\delta}}|w|_{\p}^{(2n+\delta-1)s_1+s_2-\binom{2n+\delta}{2}-2}d\mu\right)\nonumber \\
 &=\frac{1-q^{\binom{2n+\delta-1}{2}-(2n+\delta-1)s_1-s_2}}{(1-q^{\binom{2n+\delta}{2}-s_2})(1-q^{\binom{2n+\delta}{2}+1-(2n+\delta-1)s_1-s_2})},
\end{align*}
proving Theorem~\ref{cclvl} for groups of type $F$.

\begin{rmk}\label{minimality}
Formula~\eqref{kzetaF} reflects the $K$-minimality of $\Lambda=\Fnd$; see~\cite[Lemma~6.2 and~Definition~6.3]{ro17}.
In fact, the proof of Lemma~\ref{simpleF} shows in particular that
 \[\frac{\|F_{k}(A_{\Fnd}(\underline{x}))\cup wF_{k-1}(A_{\Fnd}(\underline{x}))\|_{\p}}{\|F_{k-1}(A_{\Fnd}(\underline{x}))\|_{\p}}=\|\underline{x},w\|_{\p}.\]
The formula for the local factors of the class number zeta function of $F_{n,\delta}(\ri)$ given in Corollary~\ref{kzf} coincides with the formula for the class number zeta 
function of $\Fnd(\lri)$ given by the specialisation of the formula given in~\cite[Proposition~6.4]{ro17} to the corresponding class number zeta function; 
see~\cite[Remark~4.10 and Lemma~4.11]{PL18}.
\end{rmk}

\subsection{Conjugacy class zeta functions of groups of type \texorpdfstring{$G$}{G}}\label{typeG}
In this section, we prove Theorem~\ref{cclvl} for groups of type~$G$. Analogously to the previous section, we use the descriptions~\eqref{commG} of the submatrices~$\commGm$ to determine the integrands of~\eqref{terms}, which allow us to explicitly calculate such $\p$-adic integrals.

We first describe the determinant of a square matrix in terms of its $2\times 2$-minors, which will be used to describe the terms~\eqref{terms} for groups of type~$G$. 
For a matrix $M=(m_{ij})$, denote 
\[{M}_{(i,j),(r,s)}=\left| \begin {array}{cc} m_{ir}&m_{is}\\ \noalign{\medskip}m_{jr}&m_{js}\end {array} \right|.\]

\begin{lem}\label{Minors}
Fix $t \in \N$. Let $G$ be a $2t \times 2t$-matrix and let $U$ be a $(2t+1) \times (2t+1)$-matrix with entries
\[g_{ij}=g(\underline{X})_{ij}, \hspace{0.5cm} u_{ij}=u(\underline{X})_{ij} \in \lri[\underline{X}], \] respectively.
Let $\mathbf{i}=\{i_1, \dots, i_t\}$, $\mathbf{j}=\{j_1,\dots, j_t\}\subset~[2t]$.
Then, for suitable $\alpha_{\mathbf{i},\mathbf{j}},\beta_{\mathbf{i}, \mathbf{j}} \in \{-1,1\}$, it holds that
\begin{align*}	\det(G)&=\sum_{\substack{\mathbf{i}\cup\mathbf{j}=[2t]\\ i_q<j_q,~ \forall q\in[t]}}\alpha_{\mathbf{i},\mathbf{j}}~
  {G}_{(1,2),(i_1,j_1)}{G}_{(3,4),(i_2,j_2)}\cdots{G}_{(2t-1,2t),(i_t,j_t)},
  \\
  \det(U)&=\sum_{k=1}^{2t+1}\sum_{\substack{\mathbf{i}\cup\mathbf{j}=[2t+1]\setminus\{k\}\\ i_q<j_q,~ \forall q\in[t]}} 
  \beta_{\mathbf{i},\mathbf{j}}u_{1k}~
  {U}_{(1,2),(i_1,j_1)}^{k}{U}_{(3,4),(i_2,j_2)}^{k}\cdots {U}_{(2t-1,2t),(i_t,j_t)}^{k},
\end{align*}
where $U^k$ is the matrix obtained by excluding the first row and the $k$th column of~$U$.
\end{lem}

\begin{proof}
Let $G^k=(\widetilde{g}_{ij})_{ij}$ be the matrix obtained from $G$ by excluding its first row and its $k$th column. Then, its entries are
\[\widetilde{g}_{ij}=\begin{cases}
                      g_{(i+1)j},&\text{ if }j\in[k-1],\\
                      g_{(i+1)(j+1)},&\text{ if }j\in\{k, \dots, 2t-1\}.
                 \end{cases}\]
Given two subsets $I,J \subseteq [2t]$ of equal cardinality $m$, denote by $\widehat{G}_{I,J}$ the determinant of the $(2t-m)\times (2t-m)$-submatrix of $G$ obtained by excluding the rows of indices in $I$ and columns of indices in $J$.
It follows that
\[ \det(G^k)=\sum_{j=1}^{k-1}(-1)^{1+j}g_{2j}\widehat{G}_{\{1,2\},\{j,k\}}+\sum_{j=k}^{2t-1}(-1)^{1+j}g_{2(j+1)}\widehat{G}_{\{1,2\},\{k,j+1\}}.\]
Consequently,
\begin{align*}
 \det(G)&=\sum_{k=1}^{2t}(-1)^{1+k}g_{1k}\det(G^k)\\
 &=\sum_{k=1}^{2t}\left(\sum_{j=1}^{k-1}(-1)^{k+j}g_{1k}g_{2j}\widehat{G}_{\{1,2\},\{j,k\}}-\sum_{j=k+1}^{2t}(-1)^{k+j}g_{1k}g_{2j}\widehat{G}_{\{1,2\},\{k,j\}}\right)\\
 &=\sum_{m=1}^{2t-1}\sum_{i=m+1}^{2t}(-1)^{i+m-1}(g_{1m}g_{2i}-g_{1i}g_{2m})\widehat{G}_{\{1,2\}\{m,i\}}\\
 &=\sum_{m=2}^{2t-1}\sum_{i=m+1}^{2t}(-1)^{i+m-1}{G}_{(1,2),(m,i)}\widehat{G}_{\{1,2\},\{m,i\}}.
\end{align*}
The relevant claim of Lemma~\ref{Minors} for~$G$ follows by induction on $t$.

The claim for~$U$ follows from the first part, since its determinant is  
\[\det(U)=\sum_{k=1}^{2t+1}(-1)^{k+1}u_{1k}\det(U^k).\qedhere\]
\end{proof}

\begin{lem} For each $r \in [2n]$, the nonzero elements of $F_{r}(\commG)$ are either of one of the following forms or a sum of these terms.
\[X_{i_1}\dots X_{i_{\omega}}X_{n+j_{1}}\dots X_{n+j_{\lambda}}\text{ or }-X_{i_1}\dots X_{i_{\omega}}X_{n+j_{1}}\dots X_{n+j_{\lambda}}.\]
\end{lem}
\begin{proof}
Lemma~\ref{Minors} describes each element of $F_{k}(\commG)$ in terms of sums of products of $2\times 2$-minors of $\commG$.
It then suffices to show that these minors are all either $0$ or of the forms $X_iX_j$ or $-X_iX_j$, for some $i,j\in [2n]$.
This can be seen from the description of $\commG$ in terms of the blocks \eqref{commG}.
\end{proof}

The main idea of the proof of Theorem~\ref{cclvl} for groups of type $G$ is showing the following proposition.
\begin{pps}\label{claim} Let $\underline{X}=(X_1,\dots,X_{2n})$ be a vector of variables. Consider $\lambda,\omega \in [n]_0$ such that $0 <\omega+\lambda < 2n$. Then,  
for all choices of $i_1$, $\dots$, $i_{\omega}$, $j_1$, $\dots$, $j_{\lambda} \in [n]$, one of
\[X_{i_1}\dots X_{i_{\omega}}X_{n+j_{1}}\dots X_{n+j_{\lambda}}\text{ or }-X_{i_1}\dots X_{i_{\omega}}X_{n+j_{1}}\dots X_{n+j_{\lambda}}\]
is an element of 
$F_{\omega+\lambda}(\commG)$.
\end{pps}

In fact, assuming Proposition~\ref{claim}, we have the following.  For $x,y \in \lri$, 
\[\min\{v_{\p}(x+y),v_{\p}(x),v_{\p}(y)\}=\min\{v_{\p}(x),v_{\p}(y)\}.\] Thus, if some term of the form 
\[X_{i_1}X_{i_2}\dots X_{i_{\omega}}X_{n+j_{1}}\dots X_{n+j_{\lambda}}\pm X_{k_1}X_{k_2}\dots X_{k_{\omega}}X_{n+l_{1}}\dots X_{n+l_{\lambda}}\]
is the determinant of a submatrix of~$\commG$, then, assuming that the claim in Proposition~\ref{claim} holds, (up to sign) both 
\[X_{i_1}X_{i_2}\dots X_{i_{\omega}}X_{n+j_{1}}\dots X_{n+j_{\lambda}} \text{ and }X_{k_1}X_{k_2}\dots X_{k_{\omega}}X_{n+l_{1}}\dots X_{n+l_{\lambda}}\]
are also determinants of submatrices of~$\commG$, and hence, when considering these three terms, only the last two 
will be relevant in order to determine $\|F_{k}(\commG)\|_{\p}$.
In this case, we may assume that all elements of~$F_{k}(\commG)$ and $F_{k-1}(\commG)$ are of the forms given in Proposition~\ref{claim} while computing $\|F_{k}(\commG)\|_{\p}$ and $\|F_{k}(\commG)\cup wF_{k-1}(\commG)\|_{\p}$.

\vspace{0.1cm}Firstly we show Proposition~\ref{claim} for most cases, excluding only the case $\omega=n$ and all $i_1, \dots, i_n$ distinct and the case $\lambda=n$ and all $j_1, \dots, j_n$ distinct. 
\begin{lem}\label{mostcases}
 Let $\omega,\lambda \in [n]_0$ be such that $0<\omega + \lambda <2n$. Given $i_1$, $\dots$, $i_{\omega} \in [n]$ and $j_1$, $\dots$, $j_{\lambda}\in[n]$ satisfying  $\{i_1, \dots, i_{\omega}\}\neq [n]$ and $\{j_1, \dots,j_{\lambda}\}\neq [n]$, we have that either 
 \[X_{i_1}\cdots X_{i_{\omega}}X_{n+j_1}\cdots X_{n+j_{\lambda}}~\text{ or }~-X_{i_1}\cdots X_{i_{\omega}}X_{n+j_1}\cdots X_{n+j_{\lambda}}\] 
 is an element of $F_{\lambda+\omega}(\commG)$.
\end{lem}
\begin{proof}
For each $(\mathbf{i},\mathbf{j})=(i_1, \dots, i_{\omega},j_1, \dots, j_{\lambda})$ as in the assumption of Lemma~\ref{mostcases}, we construct explicitly a submatrix of $\commG$ 
which is of the form
\begin{equation}\label{most}\begin{bmatrix}[ccc|ccc]
			X_{n+j_1} &  &  & \multicolumn{3}{c}{\multirow{3}{*}{$V(\underline{X})$}}\\
			&\ddots & & & \\
			&  & X_{n+j_{\lambda}} &  & \\
			\hline
			\multicolumn{3}{c|}{\multirow{3}{*}{$W(\underline{X})$}} & -X_{i_1}&&\\
			&&&&\ddots&\\
			&&&&&-X_{i_{\omega}}
			 
\end{bmatrix}\end{equation}
where $V(\underline{X})=(v(\underline{X})_{ij})$ and $W(\underline{X})=(w(\underline{X})_{ij})$ are such that $v(\underline{X})_{ij}=0$ and $w(\underline{X})_{ij}=0$, whenever $i \leq j$.
It is clear that the determinant of this matrix is one of $\pm X_{i_1}\cdots X_{i_{\omega}}X_{n+j_1}\cdots X_{n+j_{\lambda}}$.

The main fact we use is that the columns of $A_{\mathcal{G}_n}(\underline{X})$ are of the form \\
\begin{equation}\label{columng}
 \begin{matrix}
\vphantom{a}\\
    \coolleftbrace{\phantom{(n+j)}i\text{th row}}{X_{n+j}}\\
    \vphantom{a}\\
    \vphantom{b}\\
    \coolleftbrace{\phantom{i}(n+j)\text{th row}}{i}\\
    \vphantom{a}
\end{matrix}
\begin{bmatrix}
       \\
      X_{n+j} \\
       \\
       \\
      -X_i  \\
      \\ 
\end{bmatrix},
\end{equation}
where the nondisplayed terms equal zero and $i,j \in [n]$.
For each $i,j\in[n]$, there is exactly one column of $\commG$ with $X_{n+j}$ in the $i$-th row, and exactly one column with $-X_i$ in the $(n+j)$th row.

Firstly, we choose rows and columns of~$\commG$, so that the submatrix consisting of these rows and columns is $\diag(X_{n+j_1}, \dots, X_{n+j_{\lambda}})$, the diagonal matrix with main diagonal entries $X_{n+j_1}, \dots, X_{n+j_{\lambda}}$. 
Fix $r_1 \in [n]\setminus\{i_1,\dots, i_{\omega}\}$ and, inductively, fix $r_t \in [n]\setminus \{r_1,\dots,r_{t-1},i_t,\dots,i_{\omega}\}$ for each $t \in \{2, \dots, \lambda\}$. If $\lambda>\omega$, fix $r_t \in [n] \setminus \{r_1, \dots, r_{t-1}\}$ for $t > \omega$. 
For each $t \in [\lambda]$, let $c_t$ be the index of the unique column of~$\commG$ with $X_{n+j_t}$ in the $r_t$th row.

The submatrix consisting of rows $r_1, \dots, r_{\lambda}$ and of columns $c_1, \dots, c_{\lambda}$ in this order is $\diag(X_{n+j_1}, \dots, X_{n+j_{\lambda}})$ because the only nonzero entries of $c_t$ in~$\commG$ are the ones of indices $r_t$ and $n+j_t$. Since the chosen $\lambda \times \lambda$-matrix consists only of rows of~$\commG$ of index in~$[n]$, it follows that the only nonzero entry of $c_t$ in this submatrix is $X_{n+j_t}$. 

Now, we choose rows and columns of $\commG$ so that the submatrix composed of such rows and columns is $\diag(-X_{i_1}, \dots, -X_{i_{\omega}})$. Fix $\row_1\in[n]\setminus \{j_1, \dots, j_{\lambda}\}$ and, inductively, fix $\row_q\in [n]\setminus \{\row_1, \dots, \row_{q-1},j_q, \dots, j_{\lambda}\}$ for each $q \in \{2, \dots, \omega\}$. If $\omega> \lambda$, fix $\row_q \in [n] \setminus \{\row_1, \dots, \row_{q-1}\}$ for $q>\lambda$. For each $q \in [\omega]$, let $\col_q$ be the index of the unique column of~$\commG$ with $-X_{i_q}$ in the $(n+\row_q)$th row. 

Analogously to the previous construction, we have that the submatrix of~$\commG$ composed of rows $n+\row_1, \dots, n+\row_{\omega}$ and columns $\col_1, \dots, \col_{\omega}$ is $\diag(-X_1, \dots, -X_{\omega})$.

We claim that the submatrix~$S(\underline{X})$ of~$\commG$ composed of rows $r_1, \dots, r_{\lambda}$, $n+\row_1, \dots, n+\row_{\omega}$, and of columns $c_1, \dots, c_{\lambda}$, $\col_1 \dots, \col_{\omega}$ in this order is of the form~\eqref{most}.

First, we must assure that the columns $c_t$ and the columns $\col_q$ are all distinct. 
From~\eqref{columng}, one sees that the columns $c_t$ and $\col_q$ are

\begin{align}\label{columng2}
 \begin{matrix}
\vphantom{a}\\
    \coolleftbrace{\phantom{(n+j)}r_t\text{th row}}{X_{n+j}}\\
    \vphantom{a}\\
    \vphantom{b}\\
    \coolleftbrace{(n+j_t)\phantom{r}\text{th row}}{i}\\
    \vphantom{a}
\end{matrix}
\begin{bmatrix}
       \coolover{c_t}{\phantom{X}}\\
      X_{n+j_t} \\
       \\
       \\
      -X_{r_t}  \\
      \\ 
\end{bmatrix}
&
\begin{matrix}
\vphantom{a}\\
    \coolleftbrace{\phantom{(n+m)}i_q\text{th row}}{X_{n+j}}\\
    \vphantom{a}\\
    \vphantom{b}\\
    \coolleftbrace{(n+\row_q)\phantom{i}\text{th row}}{i}\\
    \vphantom{a}
\end{matrix}
\begin{bmatrix}
       \coolover{\col_q}{\phantom{X}}\\
      X_{n+\row_q} \\
       \\
       \\
      -X_{i_q}  \\
      \\ 
\end{bmatrix}.
\end{align}
By construction, the indices $c_t$ are all distinct, and so are the indices $\col_q$.
If $c_t=\col_q$ for some $t \in [\lambda]$ and some $q\in[\omega]$, then we would obtain $r_t=i_q$ and $\row_q=j_t$.
First, suppose that $\lambda = \omega$. 
Since $r_t \notin \{i_t, \dots, i_{\omega}\}$, for $r_t=i_q$ to hold, we must have $t>q$. However, since $\row_q \notin \{j_q, \dots, j_{\lambda}\}$, for $\row_q=j_t$ to hold, we must have $q>t$. 
Now, assume $\lambda > \omega$. In this case, for $t \in [\omega]$, the same argument as for $\lambda=\omega$ shows that we cannot have $r_t=i_q$ and $\row_q=j_t$ simultaneously. Assume then that $t > \omega$. Then, since $\row_q \notin \{j_q, \dots, j_{\lambda}\}$, for $\row_q=j_t$ to hold, we must have $q>t$. However, $t > \omega \geq q$. Analogous arguments show that, for $\omega>\lambda$, we must have $c_t\neq \col_q$ for all $t \in [\lambda]$ and all $q \in [\omega]$.

It is clear that the indices $r_k$ and $n+\row_{q}$ are all distinct. 
It follows that the submatrix~$S(\underline{X})$ is of the form~\eqref{most} for some matrices $V(\underline{X})\in \Mat_{\lambda\times\omega}(\lri[\underline{X}])$ and $W(\underline{X})\in \Mat_{\omega\times\lambda}(\lri[\underline{X}])$. We are now left to show that, in fact, $v(\underline{X})_{ij}=0$ and $w(\underline{X})_{ij}=0$ for $i\leq j$.

Let us show this for $V(\underline{X})=(v(\underline{X})_{ij})$. Observe that $v_{tq}$ is the $(r_t,\col_{q})$-entry of $\commG$. 
Since the only nonzero entries of column $\col_q$ in $\commG$ are the ones in rows $i_q$ and $n+\row_q$, we have that $v(\underline{X})_{tq}$ can only be nonzero if $r_t=i_q$. By construction, we have that $r_t \notin \{i_t, \dots, i_{\omega}\}$. Consequently, if $t \leq q$, then $r_t \neq i_q$. It follows that $v(\underline{X})_{tq}=0$ if $t \leq q$ as desired. 
Similarly, we see that $W(\underline{X})=(w(\underline{X})_{ij})$ satisfies $w(\underline{X})_{ij}=0$ whenever $i \leq j$.
\end{proof}

\begin{Proof}\hspace{-0.2cm}\textbf{of Proposition~\ref{claim}}. 
Lemma~\ref{mostcases} shows the claim of Proposition~\ref{claim} for all cases, except for $\omega=n$ and $i_1, \dots, i_n$ all distinct, and for $\lambda=n$ and $j_1, \dots, j_n$ all distinct. Let us show the last case, the other one is analogous.

As in the proof of Lemma~\ref{mostcases}, we construct a submatrix of $\commG$ of the form~\eqref{most}. 
For $t \in [n]$, we can choose $r_t$ as in the proof of Lemma~\ref{mostcases} because $|\{i_1, \dots, i_{\omega}\}|<n$. We also set $c_t$ as in the proof of Lemma~\ref{mostcases}.

As $|\{j_1, \dots, j_n\}|=n$, we cannot choose $\row_1 \in [n]\setminus \{j_1, \dots, j_n\}$. Instead, we consider the rows $n+j_q$ of $\commG$, for $q \in [\omega]$. Denote by $\col_q$ the column of $\commG$ with $-X_{i_q}$ in the $(n+j_q)$th row.

Again, we must assure that the columns $c_t$ and $\col_q$ are all distinct. 
By construction, the indices $c_t$, for $t \in [n]$, are all distinct, and so are the indices $\col_q$, for $q\in [\omega]$. 
If the columns of indices $c_t$ and $\col_q$ coincide for some $t \in [n]$ and $q \in [\omega]$, then $r_t=i_q$ and $t=q$. However, since $r_t \notin \{i_t, \dots, i_{\omega}\}$, we can only have $r_t=i_q$ if $q<t$. 

Denote by $M(\underline{X})$ be the submatrix of $\commG$ composed of columns $c_1, \dots, c_n$, $\col_1, \dots, \col_{\omega}$ and of rows $r_1, \dots, r_n$, $n+j_1, \dots, n+j_{\omega}$ in this order. 
Then, as in Lemma~\ref{mostcases}, the submatrix $M(\underline{X})$ is of the form~\eqref{most}, however the matrix $W(T)$ is such that $w(\underline{X})_{ij}=0$ if $i\neq j$.
\end{Proof}

In particular, Proposition~\ref{claim} shows that, for each $r\in [2n]$ and each $k\in[n]$, either $X_{r}^{k}$ or $-X_{r}^{k}$ is an element of 
$F_{k}(\commG)$.
Hence, if $\mathbf{x}\in \W_{2n}$, then at least one $k\times k$-minor of $A_{\Gn}(\mathbf{x})$ has valuation zero. This gives
\begin{equation}\label{integrandg}\frac{\|F_{k}(A_{\Gn}(\mathbf{x}))\cup wF_{k-1}(A_{\Gn}(\mathbf{x}))\|_{\p}}
{\|F_{k-1}(A_{\Gn}(\mathbf{x}))\|_{\p}}=1, \text{ for all }k\in[n] \text{ and }w\in\p.\end{equation}

For $k\in \{n+1, \dots, 2n-1\}$, as explained before Lemma~\ref{mostcases}, the elements of $F_{k}(\commG)$ can be assumed to be of the form 
\[X_{i_1}\cdots X_{i_{\omega}}X_{n+j_{1}}\cdots X_{n+j_\lambda},\]
where $\omega,\lambda \in [n]_0$ satisfy $\omega+\lambda=k$, and $i_1$, $\dots$, $i_\omega$, $j_1$, $\dots$, $j_\lambda \in [n]$. 

Given $\mathbf{x}\in \W_{2n}$, set $M=v_{\p}(x_1, \dots, x_n)$ and $N=v_{\p}(x_{n+1}, \dots, x_{2n})$. Then
\begin{align*}&\left\|\bigcup_{\substack{\omega+\lambda =k\\ 0 \leq \omega, \lambda \leq n}}\{X_{i_1}\cdots X_{i_{\omega}}X_{n+j_{1}}\cdots X_{n+j_\lambda}\mid i_1, \dots, i_{\omega}, 
j_1, \dots, j_{\lambda} \in [n]\}\right\|_{\p}\\&=q^{-n\min\{M,N\}-(k-n)\max\{M,N\}}.
\end{align*}
Consequently, for $w \in \p$,
\begin{equation}\label{eqg}
\frac{\|F_{k}(A_{\Gn}(\mathbf{x}))\cup wF_{k-1}(A_{\Gn}(\mathbf{x}))\|_{\p}}{\|F_{k-1}(A_{\Gn}(\mathbf{x}))\|_{\p}}=
\begin{cases}
 \|x_1, \dots, x_n, w\|_{\p}, & \text{ if } 0=N \leq M,\\
 \|x_{n+1}, \dots, x_{2n},w\|_{\p}, & \text{ if } 0=M \leq N.\\
\end{cases} 
\end{equation}
Combining equations~\eqref{integrandg} and~\eqref{eqg} yields
\[\prod_{k=1}^{2n-1}\frac{\|F_{k}(A_{\Gn}(\mathbf{x}))\cup wF_{k-1}(A_{\Gn}(\mathbf{x}))\|_{\p}}{\|F_{k-1}(A_{\Gn}(\mathbf{x}))\|_{\p}}=
\begin{cases}
 \|x_1, \dots, x_n, w\|_{\p}^{n-1}, & \text{ if } 0=M \leq N,\\
 \|x_{n+1}, \dots, x_{2n},w\|_{\p}^{n-1}, & \text{ if } 0=N \leq M.\\
\end{cases}
\]
Consequently, the $\p$-adic integral~\eqref{ccpadic} in this case is
\begin{align*}&\int_{(w,\underline{x})\in\p\times\W_{2n}}|w|_{\p}^{(2n-1)s_1+s_2-n^2-2}\prod_{k=1}^{2n-1}
\frac{\|F_{k}(A_{\Gn}(\underline{x}))\cup wF_{k-1}(A_{\Gn}(\underline{x}))\|_{\p}^{-1-s_1}}{\|F_{k-1}(A_{\Gn}(\underline{x}))\|_{\p}^{-1-s_1}}d\mu\\
 &=2\int_{(w,x_1, \dots, x_{2n})\in \p\times\p^n\times \W_{n}}|w|_{\p}^{(2n-1)s_1+s_2-n^2-2}\|x_1, \dots,x_n,w\|_{\p}^{-(n-1)(1+s_1)}d\mu\\
 &+\int_{(w,x_1,\dots,x_{2n})\in\p\times\W_{n}\times \W_{n}}|w|_{\p}^{(2n-1)s_1+s_2-n^2-2}d\mu\\
 &=\left(1-q^{-n}+2q^{-1+(n-1)s_1}-q^{n^2-ns_1-s_2}-q^{n^2-n-ns_1-s_2}\right)\cdot\\
 &\frac{(1-q^{-1})(1-q^{-n})q^{n^2+1-(2n-1)s_1-s_2}}{(1-q^{n^2+1-(2n-1)s_1-s_2})(1-q^{n^2-ns_1-s_2})},
\end{align*}
where the first and the second integrals of the second equality are calculated in Proposition~\ref{integral2} and Proposition~\ref{integral1}, respectively.
Applying this to~\eqref{ccpadic}, we obtain
\begin{align*} &\clvlzf_{G_n(\lri)}(s_1,s_2)=\\
&\frac{(1-q^{2\binom{n}{2}-ns_1-s_2})(1-q^{2\binom{n}{2}+1-(2n-1)s_2-s_2})+q^{n^2-ns_1-s_2}(1-q^{-n})(1-q^{-(n-1)(1+s_1)})}{(1-q^{n^2-s_2})(1-q^{n^2-ns_1-s_2})(1-q^{n^2+1-(2n-1)s_1-s_2})},
\end{align*}
proving Theorem~\ref{cclvl} for groups of type $G$.

\subsection{Conjugacy class zeta functions of groups of type \texorpdfstring{$H$}{H}}\label{typeH}
In this section, we prove Theorem~\ref{cclvl} for groups of type~$H$. Analogously to the previous sections, we use the descriptions~\eqref{commH} of the submatrices~$\commHm$ to determine the integrands of~\eqref{terms} and then explicitly calculate these integrals. 

In this section, we denote by $A(\underline{X})_{ij}$ the $(i,j)$-entry of~$\commH$.

By~\eqref{commH}, each column of $\commH$ is of one of the following forms:
\noindent\begin{minipage}{0.45\textwidth}
\begin{equation}\label{H1}       
\begin{matrix}
\vphantom{X}\\
    \vphantom{X}\\
    \coolleftbrace{\phantom{(n+i)}i\text{th row}}\\
    \vphantom{X}\\
    \vphantom{X}\\
    \vphantom{X}\\
    \vphantom{X}\\
    \coolleftbrace{\phantom{i}(n+i)\text{th row}}\\
    \vphantom{X}\\
    \vphantom{X}\\
    \vphantom{a}
\end{matrix}
\begin{bmatrix}
      \\
      \\
      X_{n+i}\\
       \\
       \\
       \\
       -X_{i}\\
       \\
      \\
\end{bmatrix},
\end{equation}
\end{minipage}
\hspace{0.3cm}
\begin{minipage}{0.45\textwidth}\centering
\begin{equation}\label{H2}
 \begin{matrix}
\vphantom{X}\\
    \coolleftbrace{\phantom{(n+i)}j\text{th row}}{X_{n+j}}\\
    \vphantom{X}\\
    \coolleftbrace{\phantom{(n+j)}i\text{th row}}\\
    \vphantom{X}\\
    \vphantom{X}\\
    \coolleftbrace{\phantom{i}(n+j)\text{th row}}\\
    \vphantom{X}\\
    \vphantom{X}\\
    \coolleftbrace{\phantom{j}(n+i)\text{th row}}{i}\\
    \vphantom{a}
\end{matrix}
\begin{bmatrix}
      \\
      X_{n+i} \\
       \\
       X_{n+j}\\
       \\
      -X_{i}  \\
      \\ 
      -X_{j}\\
      \\
\end{bmatrix},
\end{equation}
\end{minipage}\\
where the nondisplayed entries equal zero and $i,j \in [n]$.
Recall that $a=\rk(\g/\z)$, $b=\rk(\g')$, and that~$\commH$ is an $a \times b$-matrix. Its columns have the following symmetry: 
    \begin{equation}\label{symmetryH}	A(\underline{X})_{(n+i)j}=	\begin{cases}
									  -X_{q},&\text{ if and only if }A(\underline{X})_{ij}=X_{n+q},\\
									  0,&\text{ if and only if }A(\underline{X})_{ij}=0,
									 \end{cases}
\end{equation}
for $i \in [n]$, $j \in [b]$ and $q \in [a]$.
For each $i \in [n]$, there is exactly one column of the form~\eqref{H1}, and the columns of type~\eqref{H2} occur exactly once for each pair $i<j$ of elements of $[n]$.

\begin{lem}\label{smallnh} For $w \in \p$, $\mathbf{x}\in\W_{2n}$ and $k\in [n]$,
\[\frac{\|F_{k}(A_{\mathcal{H}_n}(\mathbf{x}))\cup wF_{k-1}(A_{\mathcal{H}_n}(\mathbf{x}))\|_{\p}}{\|F_{k-1}(A_{\Hn}(\mathbf{x}))\|_{\p}}=1.\]
\end{lem}
\begin{proof}
Fix $m \in [n]$. We now construct a submatrix of $\commH$ whose determinant is $X_{n+m}^{n}$. 

For each $t \in [m-1]$, denote by $c_t$ the index of the unique column of $\commH$ which has $X_{n+m}$ in the row of index~$t$.
Recall that $\commHm$ is the submatrix of $\commH$ given in~\eqref{commH}. 
The $n \times n$-submatrix $L_m(\underline{X})$ of $\commH$ composed of columns $c_1$, $\dots$, $c_{m-1}$ and the columns of $\commHm$ and of rows $1, \dots, n$ is
\vspace{0.7cm}
\[
 L_m(\underline{X})=\begin{bmatrix}[c c c c| c c c c]
 \coolover{c_1}{X_{n+m}} & \coolover{c_2}{~} & & \coolover{c_{m-1}}{~} & \coolover{\commHm}{\phantom{X_m0}\phantom{_{n+}}&\phantom{X_{n}0}\phantom{_{m+1}}&\phantom{0 \dots 0}&\phantom{X|}\phantom{2n}}\\
	     & X_{n+m} 	&  &  & & & & \\
	     &	    	& \ddots &  &  &  & &\\
	     &  	&  & X_{n+m} &  &  &  &\\
     X_{n+1} & X_{n+2} 	& \dots & X_{n+m-1} & X_{n+m} & X_{n+m+1} & \dots & X_{2n}\\
	     &  	&  &  &  & X_{n+m} 	&  	 &\\
	     &  	&  &  &  &  		& \ddots &  \\
	     &  	&  &  &  &  		&	 &X_{n+m}\\
 \end{bmatrix},
\] which has determinant $X_{n+m}^{n}$, as desired.

If $\mathbf{x}\in\W_{2n}$, there exists $m_0 \in [n]$ such that the matrix $L_{m_0}(\mathbf{x})$ has maximal rank $n$. That is, for each $k \in [n]$, at least one of the 
$k\times k$-minors of $L_{m_0}(\mathbf{x})$ is a unit. Since the $k\times k$-minors of $L_{m_0}(\mathbf{x})$ are elements of $F_{k}(A_{\mathcal{H}_n}(\mathbf{x}))$, the result 
follows.
\end{proof}

We now determine the sets $F_{k}(A_{\mathcal{H}_n}(\underline{X}))$ for $k >n$. 
In the next lemma, we show that all elements of such sets are linear combinations of products involving the minors 
\[M_{ij}(\underline{X}):=X_{i}X_{n+j}-X_{j}X_{n+i}\] of the matrix
\[M(X_1, \dots, X_{2n})= \left[ \begin {array}{cccc} X_1&X_2&\dots&X_n\\ \noalign{\medskip}X_{n+1}&X_{n+2}&\dots&X_{2n}\end {array} \right]\in
\Mat_{2\times n}(\lri[X_{1},\dots,X_{2n}]).\]

\begin{lem}\label{lemcommH}
Let $k=n+l$, for some $l \in [n-1]$. Then the nonzero elements of $F_{k}(A_{\mathcal{H}_n}(\underline{X}))$ are sums of terms of the form
\[X_{f_1}\dots X_{f_{\omega}}M_{i_1j_1}(\underline{X})\dots M_{i_{\lambda}j_{\lambda}}(\underline{X}),\]
for $i_1, \dots, i_{\lambda},j_1, \dots, j_{\lambda} \in [n]$, and $f_1, \dots, f_{\omega} \in [2n]$, where $\omega+2\lambda=k$ and $\lambda\geq l$.
\end{lem}
\begin{proof} 
Lemma~\ref{Minors} describes each element of $F_{k}(\commH)$ in terms of sums of products of $2 \times 2$-minors of $\commH$.  It then suffices to show that these minors are all either $0$, $\pm X_uX_v$ or $M_{ij}(\underline{X})$, for some $u,v \in [2n]$ and $1\leq i < j \leq n$.

Let $G(\underline{X})=(G(\underline{X})_{ij})$ be a $k \times k$-submatrix of $\commG$. 
Since $k=n+l$, there are at least~$l$ pairs of rows of $G(\underline{X})$ whose indices in $\commH$ are of the form $t$ and $n+t$, for some $t \in [n]$. 
Denote by $\lambda$ the exact number of such pairs of rows occurring in $G(\underline{X})$, and assume that, for $m \in \{1, 3, \dots, 2\lambda-1\}$, the $m$th and the $(m+1)$th rows of $G(\underline{X})$ correspond, respectively, to rows of indices of the form $t$ and $n+t$ in $\commH$, for some $t \in [n]$.
In this case, the symmetry~\eqref{symmetryH} of columns of~$\commH$ assures that
 \[G(\underline{X})_{ij}=0 \text{ if and only if } G(\underline{X})_{(i+1)j}=0,\]
for all $i\in \{1, 3, \dots, 2\lambda-1\}$ and $j \in [b]$ (recall that $\commH$ is an $a \times b$-matrix).
Therefore, for $k_1,k_2\in[b]$ distinct and $m \in \{1,3, \dots, 2\lambda-1\}$, the minor 
\[{G}(\underline{X})_{(m,m+1),(k_1,k_2)}=G(\underline{X})_{mk_1}G(\underline{X})_{(m+1)k_2}-G(\underline{X})_{mk_2}G(\underline{X})_{(m+1)k_1}\] 
is either $0$ or $M_{ij}(\underline{X})$, for some 
$1 \leq i < j \leq n$, as the columns of this minor are either of the form $(0,0)^{\textup{T}}$ or $(X_{n+i},-X_{i})^{\textup{T}}$, for some $i \in [n]$.

For $i,j \in [n]$ distinct, there is at most one column of $\commG$ whose nonzero rows are the ones of indices in $\{i,j,n+i,n+j\}$, it follows that each of the remaining minors of 
$G$ are either equal to $0$ or $\pm X_iX_j$, for some $i,j \in [2n]$.
\end{proof}

Fix $k=n+l$ for some $l \in [n-1]$. We want to describe $\|F_{k}(A_{\Hn}(\mathbf{x}))\|_{\p}$ in a way that allows us to compute the integral~\eqref{ccpadic} for groups of type~$H$. 

Let $\mathbf{x}=(x_1, \dots, x_{2n})\in \W_{2n}$. Then there exists $f_0 \in [2n]$ such that $v_\p(x_{f_0})=0$. 
Fix positive integers $\omega$ and $\lambda$ satisfying $k=\omega+2\lambda$ and $\lambda \geq l$ as in Lemma~\ref{lemcommH}. Let $i_1, \dots, i_{\lambda}, j_1, \dots, j_{\lambda} \in [n]$ with $1 \leq i_t < j_t \leq n$ for all $t \in[\lambda]$. Then 
\begin{align}\|x_{f_1}\dots x_{f_{\omega}}M_{i_1j_1}(\textbf{x})\dots M_{i_{\lambda}j_{\lambda}}(\textbf{x})\|_{\p} & \leq \|x_{f_0}^{\omega}M_{i_1j_1}(\textbf{x})\dots M_{i_{\lambda}j_{\lambda}}(\textbf{x})\|_{\p} \nonumber \\
\label{f0}&=\|M_{i_1j_1}(\textbf{x})\dots M_{i_{\lambda}j_{\lambda}}(\textbf{x})\|_{\p},
\end{align}
for all choices of $f_1, \dots, f_{\omega} \in [2n]$. 

Now, let $i_0, j_0 \in [n]$ be distinct and satisfy $\|\{M_{ij}(\mathbf{x}) \mid 1 \leq i \leq j \leq n\}\|_{\p}=\|M_{i_0j_0}(\mathbf{x})\|_{\p}$, that is, the minor $M_{i_0j_0}(\textbf{x})$ has minimal $\p$-adic valuation among all $M_{ij}(\textbf{x})$.
Then it is clear that 
\begin{equation}\label{i0j0}\|M_{i_1j_1}(\textbf{x}) \dots M_{i_{\lambda}j_{\lambda}}(\textbf{x})\|_{\p} \leq \|M_{i_0j_0}(\textbf{x})\|_{\p}^{\lambda},\end{equation}
for all choices of $i_1, \dots i_{\lambda}$, $j_1, \dots, j_{\lambda} \in [n]$ with $i_q < j_q$ for $q \in [\lambda]$. 

As a consequence of~\eqref{f0} and~\eqref{i0j0}, we obtain that 
\[\|x_{f_1}\dots x_{f_{\omega}}M_{i_1j_1}(\textbf{x})\dots M_{i_{\lambda}j_{\lambda}}(\textbf{x})\|_{\p} \leq \|M_{i_0j_0}(\textbf{x})\|_{\p}^{\lambda},\]
for all choices of $i_1, \dots i_{\lambda}$, $j_1, \dots, j_{\lambda} \in [n]$ and of $f_1, \dots, f_{\omega} \in [2n]$. 

Thus if, for all $f \in [2n]$ and distinct $i,j \in [n]$, the element $X_{f}^{\omega} (M_{ij}(\underline{X}))^{\lambda}$ belongs to $F_k(A_{\Hn}(\underline{X}))$, we must have 
\begin{equation}\label{intH}\| F_{k}(A_{\Hn}(\mathbf{x})) \|_{\p}= \| M_{i_0j_0}(\textbf{x})\|_{\p}^{\lambda}= \|\{M_{ij} \mid 1 \leq i < j \leq n\}\|_{\p}^{\lambda},\end{equation} 
for the minimal choice of value of~$\lambda$. Since we require $\omega+2\lambda=k$ and $\lambda \geq l$, the minimal value is $\lambda=l$, which occurs when $\omega=k-2l$.

It then suffices to show that all terms of the form $X_{f}^{\omega}M_{ij}(\underline{X})^{l}$ are elements of $F_{k}(A_{\mathcal{H}_n}(\underline{X}))$, where $\omega=k-2l$.

\begin{pps}
Given $l \in [n-1]$, let $k=n+l$ and $\omega=k-2l$. Then, for all $f \in [2n]$ and $1 \leq i < j \leq n$, either $X_{f}^{\omega}M_{ij}(\underline{X})^l$ or $-X_{f}^{\omega}M_{ij}(\underline{X})^l$ is an element of $F_{k}(A_{\mathcal{H}_n}(\underline{X}))$.
\end{pps}
\begin{proof}
 We first show that $X_{n+f}^{\omega}M_{ij}(\underline{X})^l$ lie in $F_{k}(A_{\mathcal{H}_n}(\underline{X}))$, for all $f, i, j \in [n]$ with $i<j$. We divide the proof in the cases $f \notin \{i,j\}$ and $f \in \{i,j\}$.

\noindent
\textbf{Case f $\mathbf{\notin \{i,j\}}$.} Let us construct a submatrix of $\commH$ whose determinant is $X_{n+f}^{\omega}M_{ij}(\underline{X})^{l}$ for fixed $1 \leq i < j \leq n$ and $f \in [n] \setminus \{i,j\}$. Assume first that $\omega >1$.

Firstly, we choose~$k$ rows and~$\omega$ columns of $\commH$ such that the $k \times \omega$-submatrix of $\commH$ composed of such rows and columns is 
 
\begin{equation}\label{kbyr}
	\begin{bmatrix}[cccc]
			X_{n+f} & X_{n+r_2} & \dots & X_{n+r_m}  \\
			&X_{n+f}&&\\
			&&\ddots&\\
			&&&X_{n+f}\\\hline
			\multicolumn{4}{c}{\multirow{4}{*}{\Huge 0}}\\
			&&&\\
			&&&\\
			&&&\\
\end{bmatrix}.
\end{equation}

Fix $r_1=f$, $r_2=i$ and $r_3=j$. If $\omega > 3$, then fix inductively $r_t \in [n] \setminus \{r_1, \dots, r_{t-1}\}$, for $t \in \{4, \dots, \omega\}$. 
If $\omega=2$, that is, if we want a determinant of the form $X_{f}^{2}M_{ij}(\underline{X})^l$, we fix $\row_1=j=r_3$.
Otherwise, fix $\row_1 \in [n] \setminus \{r_1, \dots, r_{\omega}\}$.
Inductively, we set $\row_q \in [n]\setminus \{r_1, \dots, r_{\omega}, \row_1, \dots, \row_{q-1}\}$, for all $q \in \{2, \dots, l\}$. 

Let $c_t$ be the index of the unique column of $\commH$ with $X_{n+f}$ in the $r_t$th row, for all $t \in [\omega]$. We claim that the submatrix~$S(\underline{X})$ of $\commH$ composed of rows $r_1, \dots, r_{\omega}$, $\row_1$, $n+\row_1$, $\dots$, $\row_l$, $n+\row_l$ and columns $c_1, \dots, c_{\omega}$, in this order, is of the form~\eqref{kbyr}. 

In fact, column~$c_1$ is of the form~\eqref{H1} in $\commH$, so that its nonzero entries are the ones of index $f=r_1$ and $n+f$. 
The nonzero entry of index $n+f$ is not on rows $n+\mathcal{R}_q$ for any $q \in [l]$ since all $\row_q$ are distinct from~$f$. It follows that the only nonzero entry of $c_1$ in $S(\underline{X})$ is the one of the first row, corresponding to row $r_1$ in $\commH$. 

For $t \in \{2, \dots, \omega\}$, the column of $\commH$ of index $c_t$ is of the form~\eqref{H2}. Thus, its nonzero rows are the ones of index in $\{r_t, f, n+r_t, n+f\}$. Again, since $\row_q \notin \{f, r_t\}$, the rows of indices $n+r_t$ and $n+f$ are not rows of $S(\underline{X})$. Then, the only nonzero rows of $c_t$ in $S(\underline{X})$ are the $t$th and the first rows, corresponding to rows $r_t$ and $r_1=f$ of $\commH$, with entries~$X_{n+f}$ and~$X_{n+r_t}$, respectively.

Now, let $\coli_{q}$ be the unique column of $\commH$ with $X_{n+i}$ in the $\row_q$th row, and let $\colj_{q}$ be the unique column of $\commH$ with $X_{n+j}$ in the $\row_q$th row. Equivalently, $\coli_q$ and $\colj_q$ are the unique columns with $-X_{i}$, respectively $-X_j$, in row $n+\row_q$.

The submatrix of~$\commH$ composed of the rows of $S(\underline{X})$ and columns $c_1, \dots, c_{\omega}$, $\coli_{1}$, $\colj_{1}, \dots, \coli_{l}$, $\colj_{l}$ is

\begin{equation}\label{case1}
\begin{bmatrix}[cccc|cccc]
			X_{n+f} & X_{n+r_2} & \dots & X_{n+r_m} & \multicolumn{4}{c}{\multirow{4}{*}{\Huge $\ast$}} \\
			&X_{n+f}&&&&&&\\
			&&\ddots&&&&&\\
			&&&X_{n+f}&&&&\\
			\hline
			\multicolumn{4}{c|}{\multirow{4}{*}{\Huge 0}}& \multicolumn{4}{c}{\multirow{4}{*}{\Huge R(\underline{X})}}\\
			&&&&&&&\\
			&&&&&&&\\
			&&&&&&&\\
\end{bmatrix},
\end{equation}
 which has determinant $X_{n+f}^{\omega}\det(R(\underline{X}))$. 
Let us now show that \[\det(R(\underline{X}))=M_{ij}(\underline{X})^l.\] 

For $\omega \geq 3$, we claim that
\begin{equation}\label{case2}
R(\underline{X})=\begin{bmatrix}[ccccc]
			X_{n+i}&X_{n+j}&&&\\
			-X_i&-X_j&&&\\
			&&\ddots&&\\
			&&&X_{n+i}&X_{n+j}\\
			&&&-X_{i}&-X_{j}\\
\end{bmatrix}.
\end{equation}

In fact, for $\row_q \notin \{i,j\}$, we have that $\coli_{q}$ and $\colj_{q}$ are of the form~\eqref{H2} in $\commH$. 
Thus, the nonzero entries of~$\coli_q$ in $\commH$ are the ones of index in $\{i,\row_q, n+i, n+\row_q\}$, and the ones of $\colj_q$ are the ones in $\{j, \row_q, n+j, n+\row_q\}$. It follows that the only nonzero entries of $\coli_{q}, \colj_q$ in $R(\underline{X})$ are the ones corresponding to the rows of index $\row_q$ and $n+\row_q$ of $\commH$, since there is no $u \in [l]$ such that $\row_{u}\in \{i,j\}$.

Let us now determine $R(\underline{X})$ for $\omega=2$. In this case, we have $\row_1 =j$, so that $\colj_{j}$ is of the form~\eqref{H1} in $\commH$ and hence its nonzero entries in $R(\underline{X})$ are the ones of index in $\{j, n+j\}$. Column $\colj_i$ is of the form~\eqref{H2}. Its nonzero entries in $\commH$ are the ones of indices in $\{i,j,n+i,n+j\}$. Since none of the $\row_q$ is~$i$, it follows that the only nonzero entries of column $\colj_i$ in $R(\underline{X})$ are the ones corresponding to the rows of $\commH$ with indices in $\{j, n+j\}$. 

Analogously to the previous case, we have that $\colj_{q}$ is of the form~\eqref{H2} in $\commH$, for $\row_q \neq j$. Thus, the nonzero entries of $\colj_{q}$ in $\commH$ are the ones corresponding to the rows of $\commH$ of indices in $\{j,\row_q, n+j, n+\row_q\}$. It follows that 
\begin{equation}\label{case3}R(\underline{X})=
\begin{bmatrix}[cc|cc|c|cc]
			X_{n+i}&X_{n+j}&0 & X_{n+\row_2}& \dots & 0 & X_{n+\row_l}\\
			-X_i&-X_j& 0 & -X_{n+\row_2}& & 0 & -X_{\row_l} \\
			\hline 
			&&X_{n+i}&X_{n+j}&&&\\
			&&-X_{i}&-X_{j}&&&\\
			\hline
			&&&&\ddots &&\\
			\hline
			&&&&&X_{n+i}&X_{n+j}\\
			&&&&&-X_{i}&-X_{j}\\
\end{bmatrix}.
\end{equation}
Since the determinant of the matrix above is $M_{ij}(\underline{X})^l$, the result follows for $\omega=2$.

Let us now consider the case $\omega=1$, that is, let us construct a submatrix of~$\commH$ with determinant $X_fM_{ij}(\underline{X})^l$ with $f \notin \{i,j\}$.
Set $r_1=f$, $\row_1=i$ and $\row_2=j$. Inductively, fix $\row_q \in [n]\setminus \{r_1, \row_1, \dots, \row_{q-1}\}$, for all $q\in \{3, \dots, l\}$.

Denote by $c_{1}^{i}$ and $c_{1}^{j}$ the indices of the columns of $\commH$ containing, respectively, $X_{n+i}$ and $X_{n+j}$ in the $r_1$th row. Denote by $\coli_{q}$ and by $\colj_{q}$ the indices of the columns of $\commH$ containing, respectively, $X_{n+i}$ and $X_{n+j}$ in the $\row_q$th row, for each $q \in [l]$.
There are only $2l-1$ indices $\colj_{q}$ and $\colj_{q}$ in total, since $\colj_{i}=\coli_{j}$.

Similar arguments as the ones of the former cases show that the matrix composed of rows $r_1$, $\row_1$, $n+\row_1, \dots, \row_l$, $n+\row_l$ and columns $c_{1}^{i}$, $c_{1}^{j}$, $\coli_{1}$, $\colj_{1}$, $\coli_2$, $\colj_{2}, \dots, \coli_{l}$, $\colj_{l}$, in this order, is
\[
\begin{bmatrix}[cc|cc|c|cc|c|cc]
			X_{n+i}& X_{n+j} & \multicolumn{2}{c|}{\multirow{1}{*}{\Large 0}}& 0 & \multicolumn{2}{c|}{\multirow{1}{*}{\Large 0}} &  & \multicolumn{2}{c}{\multirow{1}{*}{\Large 0}} \\
 			\hline
  			X_{n+f}&0&X_{n+i}&X_{n+j}&\multicolumn{1}{c|}{\multirow{2}{*}{\Large 0}}&X_{n+\mathcal{R}_3}&0& & X_{n+\mathcal{R}_l}&0\\
  			-X_f&0&-X_{i}&-X_{j}&&-X_{\mathcal{R}_3}&0& &-X_{\mathcal{R}_l}&0\\
  			\hline
  			0&X_{n+f}&0&X_{n+i}&X_{n+j}&0&X_{n+\mathcal{R}_3}& & 0&X_{n+\mathcal{R}_l}\\
  			0&-X_{f}&0&-X_i&-X_j&0&-X_{\mathcal{R}_3}&&0&-X_{\mathcal{R}_{l}}\\
  			\hline
  			\multicolumn{2}{c|}{\multirow{2}{*}{\Large 0}}&\multicolumn{2}{c|}{\multirow{2}{*}{\Large 0}}&\multicolumn{1}{c|}{\multirow{2}{*}{\Large 0}}&X_{n+i}&X_{n+j}& &\multicolumn{2}{c}{\multirow{2}{*}{\Large 0}}\\
  			&&&&&-X_i&-X_j&&&\\
  			\hline
  			&&&&&&&\multicolumn{1}{c|}{\multirow{2}{*}{$\ddots$}}&&\\
  			&&&&&&&&&\\
  			\hline
  			\multicolumn{2}{c|}{\multirow{2}{*}{\Large 0}}&\multicolumn{2}{c|}{\multirow{2}{*}{\Large 0}}&\multicolumn{1}{c|}{\multirow{2}{*}{\Large 0}}&\multicolumn{2}{c|}{\multirow{2}{*}{\Large 0}}&&X_{n+i}&X_{n+j}\\
  			&&&&&&&&-X_{i}&-X_{j}\\
\end{bmatrix}.
\]
The determinant of such matrix is 
\[M_{ij}(\underline{X})^{l-2} \det\left(\begin{bmatrix}[cc|cc|c]
			X_{n+i}& X_{n+j} & 0&0&0\\
 			\hline
  			X_{n+f}&0&X_{n+i}&X_{n+j}&0\\
  			-X_f&0&-X_{i}&-X_{j}&0\\
  			\hline
  			0&X_{n+f}&0&X_{n+i}&X_{n+j}\\
  			0&-X_{f}&0&-X_i&-X_j\\
\end{bmatrix}
 \right)=X_{n+f}M_{ij}(\underline{X})^l.\]

\noindent
\textbf{Case f $\mathbf{\in \{i,j\}}$.} Let us construct a submatrix of $\commH$ whose determinant is $X_{f}^{\omega}M_{ij}(\underline{X})^{l}$ for fixed $1 \leq i < j \leq n$ and $f \in \{i,j\}$. 

Set $r_1=f$ and $r_2 \in \{i,j\} \setminus \{f\}$. If $\omega \geq 3$, fix inductively $r_t \in [n]\setminus \{r_1, \dots, r_{t-1}\}$ for $t \in \{3, \dots, \omega\}$. 
If $\omega=1$, set $\row_1 \in \{i,j\} \setminus \{f\}$. Otherwise, set $\row_1 \in [n]\setminus \{r_1, \dots, r_{\omega}\}$. 
Fix inductively $\row_q \in [n]\setminus \{r_1, \dots, r_{\omega}, \row_1, \dots, \row_{q-1}\}$ for each $q \in \{2, \dots,l\}$. 
For each $t \in [\omega]$, denote by $c_t$ the index of the unique column of $\commH$ containing $X_{n+f}$ in the $r_t$th row. 
Moreover, for each $q \in [l]$, denote by $\coli_{q}$ and by $\colj_{q}$ the indices of the columns of $\commH$ containing, respectively, $X_{n+i}$ and $X_{n+j}$ in the $\row_q$th row.

Let $T(\underline{X})$ be the submatrix of~$\commH$ composed of rows $r_1, \dots, r_{\omega}$, $\row_1$, $n+\row_1, \dots, \row_l$, $n+\row_l$ and of columns $c_1, \dots, c_{\omega}$, $\coli_{1}$, $\colj_{1}$, $\dots$, $\coli_{l}$, $\colj_{l}$, in this order. Using analogous arguments to the ones of the former cases, one shows that $T(\underline{X})$ is of the form~\eqref{case1} with $R(\underline{X})$ as in~\eqref{case2} for $\omega \geq 2$, and as in~\eqref{case3} for $\omega=1$.

The minors of the form $X_{f}^{m}M_{ij}(\underline{X})^l$, for $f,i,j \in [n]$, (up to sign) are obtained by repeating the constructions above for each case 
but considering rows $n+r_t$ instead of $r_t$, for all $t \in [\omega]$. The determinants of the matrices obtained in this way are of the desired form because of the symmetry~\eqref{symmetryH} of the columns of $\commH$. 
\end{proof}

We are now ready to describe the integrands of~\eqref{ccpadic} in the context of groups of type~$H$. 
For $\mathbf{x} \in \W_n$ and $l \in [n-1]$, we obtain from~\eqref{intH} that 
\begin{align}
&\frac{\|F_{n+l}(A_{\mathcal{H}_n}(\mathbf{x}))\cup w F_{n+l-1}(A_{\mathcal{H}_n}(\mathbf{x}))\|_{\p}}{\|F_{n+l-1}(A_{\mathcal{H}_n}(\mathbf{x}))\|_{\p}}\nonumber \\
&=\frac{\|\{M_{ij}(\mathbf{x})^l\mid 1 \leq i < j \leq n\}\cup w\{M_{ij}(\mathbf{x})^{l-1}\mid 1 \leq i < j \leq n\}\|_{\p}}{\|\{M_{ij}(\mathbf{x})^{l-1}\mid 1 \leq i < j \leq n\}\|_{\p}}\nonumber \\
&\label{integrandH}=\|\{M_{ij}(\mathbf{x})\mid 1 \leq i < j \leq n\}\cup \{w\}\|_{\p}.
\end{align}

Combining~\eqref{integrandH} with Lemma~\ref{smallnh} yields 
\[\prod_{k=1}^{2n-1}\frac{\|F_{k}(A_{\mathcal{H}_n}(\mathbf{x}))\cup wF_{k-1}(A_{\mathcal{H}_n}(\mathbf{x}))\|_{\p}}
{\|F_{k-1}(A_{\mathcal{H}_n}(\mathbf{x}))\|_{\p}}=\|\{M_{ij}(\mathbf{x}) \mid 1 \leq i < j \leq n \} \cup \{w\}\|_{\p}^{n-1}.\]
Thus, for groups of the form $H_{n}(\lri)$, the $\p$-adic integral~\eqref{ccpadic} is
\begin{align*}&\mathcal{J}_{H_n}(s_1,s_2):=\\
&\int_{(w,\underline{x})\in\p\times\W_{2n}}\hspace{-0.2cm}|w|_{\p}^{(2n-1)s_1+s_2-\binom{n+1}{2}-2}
\|\{M_{ij}(\underline{x}) \mid 1 \leq i < j \leq n \} \cup \{w\}\|_{\p}^{-(n-1)(1+s_1)}d\mu,
\end{align*}
which is a specialisation of the integral given in Proposition~\ref{integral3}. 
Combining Proposition~\ref{integral3} with Proposition~\ref{intpadic} yields
\begin{align*}\clvlzf_{H_n(\lri)}(s_1,s_2)&=\frac{1}{1-q^{\binom{n+1}{2}-s_2}}\left(1+(1-q^{-1})^{-1}\mathcal{J}_{H_n}(s_1,s_2)\right) \\
 &=ZF_{H_n}(q,q^{-s_1},q^{-s_2}),
\end{align*}
where $ZF_{H_n}(X,T_1,T_2) \in \Q[X,T_1,T_2]$ is given by
\[\frac{(1-X^{\binom{n}{2}}T_{1}^{n}T_2)(1-X^{\binom{n}{2}+2}T_{1}^{2n-1}T_2)+X^{\binom{n+1}{ 2}}T_{1}^{n}T_2(1-X^{-n+1})(1-X^{-(n-1)}T_{1}^{n-1})}
{(1-X^{\binom{n+1}{2}}T_2)(1-X^{\binom{n+1}{2}+1}T_{1}^{n}T_2)(1-X^{\binom{n+1}{2}+1}T_{1}^{2n-1}T_2)}.\]

This proves Theorem~\ref{cclvl} for groups of type $H$.

\section{Bivariate representation zeta functions---proof of Theorem~\ref{rlvl}}\label{prlvl}
Recall that $\g:=\Lambda(\lri)$. Consider the $B$-commutator matrix $B_{\Lambda}(\underline{Y})$ of $\g$ with respect to $\mathbf{e}$ and $\mathbf{f}$ defined in Section~\ref{ppspadic}.

Recall that a matrix $M\in \Mat_{n\times n}(\lri/\p^N)$ is said to have \emph{elementary divisor type} $(m_1, \dots, m_{u})$, denoted $\nu(M)=(m_1, \dots, m_{u})$, if it is equivalent to the matrix 
$\textup{Diag}(\pi^{m_1}, \dots, \pi^{m_u}, \mathbf{0}_{n-u})$, where $\pi$ denotes a uniformiser of $\lri$, $u$ is the rank of~$M$, $\mathbf{0}_{n-u}=(0,\dots, 0)\in \Z^{n-u}$, and $0 \leq m_1 \leq m_2 \leq \dots \leq m_{u}\leq N$. 

For $k\in\N$, $N\in \N_0$, set \[W_{k}(\lrip)=\left\{\mathbf{x}\in(\lrip)^k\mid v_{\p}(\mathbf{x})=0\right\}.\] 
Recall that $\ub=n$ for all $\Lambda \in \{\Fnd,\Gn,\Hn\}$; see Section~\ref{ppspadic}.
Write $\m=(m_1, \dots, m_{n})$. In~\cite[Proposition~4.7]{PL18}, the following is shown:
\begin{align}\label{rint} &(1-q^{r-s_2})\rlvlzf_{\G(\lri)}(s_1,s_2)=\\
&\left(1+\sum_{N=1}^{\infty}\sum_{\m\in\N_{0}^{n}}\No q^{-N(n s_1+s_2+2n-r)-2\sum_{j=1}^{n}m_j\frac{(-s_1-2)}{2}}\right),\nonumber
\end{align}
where $\No =|\{\mathbf{y} \in W_{n}(\lrip) \mid \vry= \mathbf{m}\}|$ and $r=\rk(\g/\g')$. 

Given a set $I=\{i_1, \dots, i_l\}_{<} \subseteq [n-1]_0$, recall that $\mu_j:=i_{j+1}-i_j$ for all $j \in [l]_0$, where $i_0=0$, $i_{l+1}=n$. Choose $r_{I}=(r_i)_{i \in I}\in \N^{I}$ and let $N=\sum_{i \in I}r_i$. 
Recall that $b=\rk(\g')$. 
Following~\cite[Section~3]{StVo14}, we define the following sets, which form a partition of $W_{b}(\lrip)$:
\begin{align*}\textup{N}_{I,r_{I}}(\G)=&\{\mathbf{y} \in W_{b}(\lrip):\nu(B_{\Lambda}(\mathbf{y}))\\
&=(\underbrace{0, \dots, 0}_{\mu_l \text{ terms}},\underbrace{r_{i_l}, \dots, r_{i_l}}_{\mu_{l-1} \text{ terms}},\underbrace{ r_{i_l}+r_{i_{l-1}}, \dots, r_{i_l}+r_{i_{l-1}}}_{\mu_{l-2} \text{ terms} }, \dots, \underbrace{N, \dots, N}_{\mu_0 \text{ terms}})\}.
\end{align*}
For $\Lambda \in \{\Fnd,\Gn,\Hn\}$, we have that $\z=\g'$, thus
\[r:=\rk(\g/\g')=a:=\rk(\g/\z)=\begin{cases}2n+\delta,&\text{ if }\G=F_{n,\delta},\\ 2n,&\text{ if }\G\in \{G_n,H_n\}.\end{cases}\] 
For simplicity, consider $\delta=0$ when $\G\in \{G_n,H_n\}$, so that we can write $a=2n+\delta$ uniformly. 

Using these facts, we rewrite~\eqref{rint} as follows.
\begin{align}
 &\rlvlzf_{\G(\lri)}(s_1,s_2)\nonumber\\
 &=\frac{1}{1-q^{\bar{a}(\G,n)-s_2}}\sum_{I \subseteq [n-1]_0}\sum_{r_{I}\in \N^{I}}|\textup{N}_{I,r_I}(\G)|q^{-(ns_1+s_2+2n-r)\sum_{i\in I}r_i-\sum_{i\in I}ir_i(-2-s_1)}\nonumber\\
 \label{lem1}&=\frac{1}{1-q^{\bar{a}(\G,n)-s_2}}\sum_{I \subseteq [n-1]_0}\sum_{r_{I}\in \N^{I}}|\textup{N}_{I,r_I}(\G)|q^{\sum_{i\in I}r_i(-(n-i)s_1-s_2+2i+\delta)},
 \end{align}
where $\bar{a}(\G,n)=2n+\delta$, as in Theorem~\ref{rlvl}.

The cardinalities $|\textup{N}_{I,r_I}(\G)|$ are described in~\cite[Proposition~3.4]{StVo14} in terms of the polynomials $f_{\G,I}$ and the numbers $\bar{a}(\G,i)$ 
defined in Theorem~\ref{rlvl} as follows. 
\begin{align}\label{rec} |\textup{N}_{I,r_I}(\G)|&=f_{\G,I}(q^{-1})q^{\sum_{i \in I}r_i({a}(\G,i)-2i-\delta)}.
\end{align}
Combining~\eqref{rec} with~\eqref{lem1} yields
\begin{align*}\rlvlzf_{\G(\lri)}(s_1,s_2)&=\frac{1}{1-q^{\bar{a}(\G,n)-s_2}}\sum_{I \subseteq [n-1]_0}\sum_{r_{I}\in \N^{I}}f_{\G,I}(q^{-1})q^{\sum_{i \in I}r_i(\bar{a}(\G,i)-(n-i)s_1-s_2)}\\
 &=\frac{1}{1-q^{\bar{a}(\G,n)-s_2}}\sum_{ I \subseteq [n-1]_0}f_{\G,I}(q^{-1})\prod_{i\in I}\frac{q^{\bar{a}(\G,i)-(n-i)s_1-s_2}}{1-q^{\bar{a}(\G,i)-(n-i)s_1-s_2}}.\\
\end{align*}
This concludes the proof of Theorem~\ref{rlvl}.
\section{Hyperoctahedral groups and functional equations}
In this section, we relate the formulae of Theorem~\ref{rlvl} to statistics on Weyl groups of type $B$, also called hyperoctahedral groups $B_n$.
Specialisation~\eqref{specialisation} then provides formulae for the class number zeta functions of groups of type $F$, $G$, and $H$ in terms of 
such statistics. By comparing these formulae to the ones of Corollary~\ref{kzf}, we obtain formulae for joint distributions of three functions on such Weyl groups.

We also use the descriptions of the bivariate representation zeta functions in terms of Weyl group statistics in order to prove Theorem~\ref{functeq} in Section~\ref{pfuncteq}.

Some required notation regarding hyperoctahedral groups is given in Section~\ref{Weyl}.
\subsection{Hyperoctahedral groups \texorpdfstring{$B_n$}{Bn}}\label{Weyl}
We briefly recall the definition of the hyperoctahedral groups $B_n$ and some statistics associated to them. For further details about Coxeter groups and hyperoctahedral 
groups we refer the reader to~\cite{BjBr05}.

Weyl groups of type $B$ are the groups $B_n$, for $n\in\N$, of all bijections $w: [\pm n] \to  [\pm n]$ with $w(-i)=-w(i)$, for all $i \in [\pm n]$, with operation given 
by composition. 
Given an element $w \in B_n$ write $w=[a_1, \dots, a_n]$ to denote $w(i)=a_i$. 

\begin{dfn}  For $w \in B_n$, the \emph{inversion number}, the \emph{number of negative entries} and the \emph{number of negative sum pairs} of $w$ are defined, respectively, by
\begin{align*} \invb(w)&=|\{(i,j)\in [n]^2 \mid i<j, w(i)>w(j)\}|,\\
\negb(w)&=|\{i\in [n] \mid w(i)<0\}|, \\
\nspb(w)&=|\{(i,j)\in [n]^2: i\neq j, w(i)+w(j)<0\}|.\\
\end{align*}
\end{dfn}\vspace{-0.5cm}
Let $s_i=[1, \dots, i-1, i+1, i, \dots, n]$ for $i\in [n-1]$ and $s_0=[-1, 2, \dots, n]$ be elements of $B_n$. Let $S_{B}=\{s_0, s_1, \dots, s_{n-1}\}$. Then $(B_n,S_B)$ is a Coxeter system.

In~\cite[Proposition~8.1.1]{BjBr05} it is shown that the \emph{Coxeter length} on $B_n$ with respect to the generating set $S_{B}$ is  
\[\ell(w)=\invb(w)+\negb(w)+\nspb(w),\text{ for }w \in B_n.\]

The \emph{right descent} of $w \in B_n$ is the set 
\[D(w)=\{i \in [n-1]_0 \mid w(i)>w(i+1)\}.\] 
Moreover, for $I \subseteq [n-1]_0$, define 
\[B_{n}^{I}=\{w \in B_n \mid D(w)\subseteq I^{\komplement}=[n-1]_0\setminus I\}.\]\vspace{-0.6cm} 
\begin{exm} The longest element of $B_n$ is $w_o=[-1,\dots,-n]$. We have that 
 \[\invb(w_0)=\binom{n}{2}, \hspace{0.5cm} \negb(w_0)=n, \hspace{0.5cm} \ell(w_0)=n^2, \hspace{0.5cm} D(w_0)=[n-1]_0.\qedhere\]
\end{exm}
Consider $w\in B_n$. The following statistics are used in the current paper. 
\begin{align*} 
L(w)&= \frac{1}{2}|\{(i,j) \in [\pm n]_{0}^{2} \mid i<j , w(i)>w(j), i \not\equiv 0 \bmod 2\}|,\\
\des(w)&=|D(w)|,\\
\sigma(w)&=\sum_{i\in D(w)}n^2-i^2,\\
\maj(w)&=\sum_{i\in D(w)}i,\\
\rmaj(w)&=\sum_{i\in D(w)}n-i. 
\end{align*}
The statistics $\des(w)$, $\maj(w)$, and $\rmaj(w)$ are called, respectively, the \emph{descent number}, the \emph{major index}, and the \emph{reverse major index} of $w$.
\subsection{Bivariate representation zeta functions and statistics of Weyl groups}\label{stat}
The following lemma describes the polynomials $f_{\G,I}$ defined in Theorem~\ref{rlvl} in terms of statistics on the groups $B_n$, where $\G\in\{F_{n,\delta},G_n,H_n\}$.
\begin{lem}\label{fGi} Let $n \in \N, \delta \in \{0,1\}$ and $I\subseteq [n-1]_0$. Then
\begin{enumerate} \item~\cite[Proposition~4.6]{StVo14} 
\begin{align*}f_{F_{n,\delta},I}(X)&= \sum_{w \in B_{n}^{I^{\komplement}}}(-1)^{\negb(w)}X^{(2\ell(w)+(2\delta-1)\negb)(w)},\\
f_{G_n,I}(X)&= \sum_{w \in B_{n}^{I^{\komplement}}}(-1)^{\negb(w)}X^{\ell(w)},
\end{align*}
\item~\cite[Theorem~5.4]{Ca10} \begin{align*}f_{H_n,I}(X)= &\sum_{w \in B_{n}^{I^{\komplement}}}(-1)^{\ell(w)}X^{L(w)}.\phantom{X^{(2(2\delta-1))\negb(w)}}\end{align*}
\end{enumerate}
\end{lem}

\begin{lem}\label{weylfgh} Given $n \in \N$, $\delta \in \{0,1\}$, and a prime ideal $\p$ of $\ri$,
\begin{equation}\label{weylfgheq}\rlvlzf_{\G(\lri)}(s_1,s_2)=
\frac{\sum_{w \in B_n} \chi_{\G}(w)q^{-h_{\G}(w)} \prod_{i \in D(w)}q^{\bar{a}(\G,i)-(n-i)s_1-s_2}}{\prod_{i =0}^{n}(1-q^{\bar{a}(\G,i)-(n-i)s_1-s_2})},\end{equation}
where, for each $w\in B_n$,
\begin{table*}[h]\caption {}
	\centering
		\begin{tabular}{ c | c | c }
			$\G$ & $\chi_{\G}(w)$ & $h_{\G}(w)$  \\[0.5ex]
				\hline
				$F_{n,\delta}$ & $(-1)^{\negb(w)}$ & $2\ell(w)+(2\delta-1)\negb(w)$ \\
				
				$G_n$ & $(-1)^{\negb(w)}$ & $\ell(w)$ \\
				
				$H_n$ & $(-1)^{\ell(w)}$ & $L(w)$ \\
		\end{tabular}
\end{table*}
\end{lem}
\begin{proof}  
Applying Lemma~\ref{fGi} to the formulae of Theorem~\ref{rlvl} yields the following expression for $\rlvlzf_{\G(\lri)}(s_1,s_2)$:
\[\frac{1}{1-q^{\bar{a}(\G,n)-s_2}}\sum_{I\subseteq[n-1]_0}  \sum_{w \in B_{n}^{I^{\komplement}}}\chi_{\G}(w)q^{-h_{\G}(w)}\prod_{i \in I}
\frac{q^{\bar{a}(\G,i)-(n-i)s_1-s_2}}{1-q^{\bar{a}(\G,i)-(n-i)s_1-s_2}},\]
which can be rewritten as the claimed sum because of~\cite[Lemma~4.4]{StVo14}.\qedhere
\end{proof}

\begin{pps}\label{statF} For $n \in \N$ and $\delta \in \{0,1\}$, the following holds in $\Q[X,Z]$.
 \begin{align*}	&\sum_{w \in B_n} (-1)^{\negb(w)}X^{-(2(\ell-\sigma)+(2\delta-1)\negb-(2\delta-3)\rmaj-(2n+\delta)\des)(w)}Z^{\des(w)}&\\
		&=\left(1-X^{\binom{2n+\delta-1}{2}}Z\right)\prod_{i =2}^{n}\left(1-X^{\binom{2n+\delta}{2}-\binom{2i+\delta}{2}+2i+\delta}Z\right)&
 \end{align*}
\end{pps}
\begin{proof}
On the one hand, specialisation~\eqref{specialisation} applied to~\eqref{weylfgheq} for groups of type $F$ gives
 \[\kzeta_{F_{n,\delta}(\lri)}(s)=\frac{ \sum_{w \in B_n} (-1)^{\negb(w)}q^{-(2\ell+(2\delta-1)\negb)(w)} 
 \prod_{i \in D(w)}q^{\bar{a}(F_{n,\delta},i)-s}}{\prod_{i=0}^{n}(1-q^{\bar{a}(F_{n,\delta},i)-s})}.\]\\
Now,  
\[\bar{a}(F_{n,\delta},i)=\tbinom{2n+\delta}{2}-\tbinom{2i+\delta}{2}+2i+\delta=2(n^2-i^2)+(2\delta-3)(n-i)+2n+\delta,\]
so that
\[\prod_{i\in D(w)}q^{\bar{a}(F_{n,\delta},i)-s}=q^{(2\sigma+(2\delta-3)\rmaj+(2n+\delta-s)\des)(w)}.\]
Hence
\[\kzeta_{F_{n,\delta}(\lri)}(s)=\frac{ \sum_{w \in B_n} (-1)^{\negb(w)}q^{-(2(\ell-\sigma)+(2\delta-1)\negb-(2\delta-3)\rmaj-(2n+\delta-s)\des)(w)}}
{\prod_{i=0}^{n}(1-q^{\bar{a}(F_{n,\delta},i)-s})}.\]
On the other hand, Corollary~\ref{kzf} asserts that
\[\kzeta_{F_{n,\delta}(\lri)}(s)=\frac{1-q^{\binom{2n+\delta-1}{2}-s}}{(1-q^{\binom{2n+\delta}{2}+1-s})(1-q^{\binom{2n+\delta}{2}-s})}=
\frac{1-q^{\binom{2n+\delta-1}{2}-s}}{\prod_{i=0}^{1}(1-q^{\bar{a}(F_{n,\delta},i)-s})}.\]
Therefore
\begin{align*}&\sum_{w \in B_n} (-1)^{\negb(w)}q^{-(2(\ell-\sigma)+(2\delta-1)\negb-(2\delta-3)\rmaj-(2n+\delta-s)\des)(w)}&\\
  &=\left(1-q^{\binom{2n+\delta-1}{2}-s}\right)\prod_{i =2}^{n}\left(1-q^{\bar{a}(F_{n,\delta},i)}q^{-s}\right)&\\
  &=\left(1-q^{\binom{2n+\delta-1}{2}-s}\right)\prod_{i =2}^{n}\left(1-q^{\binom{2n+\delta}{2}-\binom{2i+\delta}{2}+2i+\delta}q^{-s}\right).&
\end{align*}
The formal identity follows as these formulae hold for all prime powers $q$ and all $s \in \C$ with sufficiently large real part.
\end{proof}
For a geometric interpretation of $\ell-\sigma$, we refer the reader to \cite[Section~2]{StWa98}.

It can be easily checked that, for $n\geq 2$ and $w \in B_n$,
\begin{align}
 \label{statG}\prod_{i \in D(w)}q^{\bar{a}(G_n,i)-s}&=q^{(\sigma+2\maj-s\des)(w)},\\
 \label{statH}\prod_{i \in D(w)}q^{\bar{a}(H_n,i)-s}&=q^{\frac{1}{2}(\sigma-3\rmaj)(w)+(2n-s)\des(w)}.
\end{align}
The following proposition follows from Lemma~\ref{weylfgh}, Corollary~\ref{kzf}, equalities~\eqref{statG} and~\eqref{statH}, and arguments analogous to those given in 
the proof of Proposition~\ref{statF}.
\begin{pps}\label{statGH} For $n\geq 2$, the following identities hold in $\Q[X,Z]$.
\begin{align*}
  &\sum_{w \in B_n} (-1)^{\negb(w)}X^{-(\ell-\sigma-2\maj)(w)}Z^{\des(w)}=\left(\prod_{i=3}^{n}1-X^{n^2-i^2+2i}Z\right)\cdot\\
  &\left((1-X^{2\binom{n}{2}}Z)(1-X^{2\binom{n}{2}+1}Z)+X^{n^2}Z(1-X^{-n})(1-X^{-n+1})\right), \text{ and }&\\
  &\sum_{w \in B_{n}}(-1)^{\ell(w)}X^{-\frac{1}{2}\left(2L-\sigma+3\rmaj+4n\des\right)(w)}Z^{\des(w)}=\left(\prod_{i=3}^{n}1-X^{\binom{n+1}{2}-\binom{i+1}{2}+2i}Z\right)\cdot
  \\&\left((1-X^{\binom{n}{2}}Z)(1-X^{\binom{n}{2}+2}Z)+X^{\binom{n+1}{2}}Z(1-X^{-n+1})^2\right).
 \end{align*}
\end{pps}
\begin{rmk}
 By setting $X=1$ in the equations of Propositions~\ref{statF} and~\ref{statGH}, we obtain the the equalities
 \[  \sum_{w \in B_n}(-1)^{\negb(w)}Z^{\des(w)}=\sum_{w \in B_n}(-1)^{\ell(w)}Z^{\des(w)}=(1-Z)^n,\]
which were first proven in~\cite[Theorem~3.2]{Re95}.
\end{rmk}
\subsection{Functional equations---proof of Theorem~\ref{functeq}}\label{pfuncteq}
We recall that the formulae of Proposition~\ref{intpadic} of the local factors of the bivariate representation zeta function of groups of type $F$, $G$, and $H$ hold for all nonzero prime ideals $\p$, since we consider the construction of the unipotent group schemes of class $2$ given in \cite[Section~2.4]{StVo14}. 
In particular, Lemma~\ref{weylfgh} holds for all nonzero prime ideals.
We use this fact to show that all local terms of these bivariate zeta functions satisfy functional equations. 
Recall that $w_0$ denotes the longest element of $B_n$, that is,  $w_0=[-1,-2,\dots,-n]$.

Theorem~\ref{functeq} follows from the same arguments of the proof of~\cite[Theorem~2.6]{KlVo09} applied to~\eqref{weylfgheq}. In fact, although $h_{\G}$ is not one of the statistics $\mathbf{b\cdot l_{L}}$ or $\mathbf{b\cdot l_{R}}$ defined in~\cite[Theorem~2.6]{KlVo09}, it satisfies the equations (2.6) of~\cite{KlVo09}, that is, 
\[h_{\G}(ww_0)+h_{\G}(w)=h_{\G}(w_0).\]
In fact, one can easily show that $g\in \{\invb, \negb, \ell\}$ satisfies $g(ww_0)=g(w_0)-g(w)$, for all $w \in B_n$, and 
the equation $L(ww_0)=L(w_0w)=L(w_0)-L(w)$ is~\cite[Corollary~7]{StVo13}.
Therefore the conclusion of~\cite[Theorem~2.6]{KlVo09} also holds for~\ref{weylfgheq}.

 \section*{Acknowledgments}
 This paper is part of my PhD thesis. I am grateful to my advisor Christopher Voll for his guidance, encouragement, constructive criticism and helpful discussions.
 I would like to express my gratitude to Tobias Rossmann for helpful conversations and significant comments on a previous version of this paper.
 I would like to thank Yuri Santos Rego for his support and for his comments on an earlier draft of this paper. 
 I am grateful to the anonymous referee for helpful comments.
 I also gratefully acknowledge financial support from the DAAD for this work.

\section*{Note added in proof} In a recent article~\cite{RoVo19}, Rossmann and Voll investigate zeta functions enumerating conjugacy classes of unipotent groups associated with graphs.
As an application of their results, they give alternative proofs of Corollary~\ref{kzf} for groups of type~$F$ and~$G$; cf. \cite[Examples 8.3(i) and~8.8(iii)]{RoVo19}. They also consider the bivariate conjugacy class zeta functions of a family of groups associated to so-called cographs. This class of groups contains all groups of type~$F$ and~\cite[Example~8.23]{RoVo19} provides an alternative proof of the first part of Theorem~\ref{cclvl}; cf. \cite[Example~8.23]{RoVo19}.

\def\cprime{$'$} \def\cprime{$'$}
 \providecommand{\bysame}{\leavevmode\hbox to3em{\hrulefill}\thinspace}
 \providecommand{\MR}{\relax\ifhmode\unskip\space\fi MR }

\printbibliography

\end{document}